\documentclass[a4paper]{amsart}

\usepackage{hyperref}
\usepackage{enumerate}
\usepackage{array} 
\usepackage{amsmath}             
\usepackage{amssymb}              
\usepackage{lmodern,bm}              
\usepackage{longtable}
\usepackage{tikz}
\usepackage[all]{xy}
\usepackage{dsfont}
\usepackage{pdflscape}
\usepackage{bigdelim}
\usepackage{multirow}
\usepackage{wasysym}
\usepackage{adjustbox}
\usepackage{booktabs}
\usetikzlibrary{decorations}
\usetikzlibrary{snakes}
\usetikzlibrary{matrix}
\usetikzlibrary{backgrounds}

\allowdisplaybreaks
\newtheorem{theorem}{Theorem}[section]
\newtheorem{introthm}{Theorem}
\newtheorem{lemma}[theorem]{Lemma}
\newtheorem{proposition}[theorem]{Proposition}

\theoremstyle{definition}

\newtheorem{algorithm}[theorem]{Algorithm}

\newtheorem{example}[theorem]{Example}
\newtheorem{remark}[theorem]{Remark}

\numberwithin{equation}{theorem}

\newcommand{\myplus}{\mkern-5mu+\mkern-5mu}

\newcommand{\mynus}{\mkern-5mu-\mkern-5mu}

\newcommand{\lm}{\mathrm{lm}}
\newcommand{\Cov}{\mathrm{Cov}}

\newcommand{\BigWedge}{\Lambda}

\newcommand{\iE}[1]{\mathtt{#1}}
\newcommand{\iiE}[1]{\boldsymbol{#1}}
\newcommand{\iiiE}[1]{\mathfrak{#1}}
\newcommand{\ivE}[1]{\mathit{#1}}

\newcommand{\onee}[2]{
\draw[line width=0.5pt, draw=black] (#1,0)-- (#1,-0.38);
\node[font=\tiny,align=center] at (#1,-0.5) {$\iE{#2}$};
\draw[black, fill=black] (#1,0) circle [radius=1.5pt];
}

\newcommand{\oneeu}[2]{
\draw[line width=0.5pt, draw=black] (#1,0)-- (#1,0.38);
\node[font=\tiny,align=center] at (#1,0.5) {$\iE{#2}$};
\draw[black, fill=black] (#1,0) circle [radius=1.5pt];
}

\newcommand{\oneel}[2]{
\draw[line width=0.5pt, draw=black] (#1,0)-- (#1-0.35,0);
\node[font=\tiny,align=center] at (#1-0.47,0) {$\iE{#2}$};
\draw[black, fill=black] (#1,0) circle [radius=1.5pt];
}

\newcommand{\twoe}[2]{
\begin{scope}
\clip (#1-0.2,-0.2) rectangle (#1+0.2,0.6);
\draw[line width=0.5pt, draw=black]  (#1,0) .. controls (#1-0.6,0.5) and (#1+0.6,0.5) .. (#1,0) ;
\node[font=\tiny,align=center] at (#1,0.5) {$\iiE{#2}$};
\draw[black, fill=black] (#1,0) circle [radius=1.5pt];
\end{scope}
}

\newcommand{\twostr}[3]{
\begin{scope}
\draw[line width=0.5pt, draw=black]  (0,0) .. controls (#1-#2,0.6) and (#1+#2,0.6) .. (0,0) ;
\node[font=\tiny,align=center] at (#1,0.6) {$\iiE{#3}$};
\draw[black, fill=black] (0,0) circle [radius=1.5pt];
\end{scope}
}

\newcommand{\twoes}[3]{
\begin{scope}
\clip (#1-0.2,-0.2) rectangle (#1+0.2,0.6);
\draw[line width=0.5pt, draw=black]  (#1,0) .. controls (#1-#3,0.5) and (#1+#3,0.5) .. (#1,0) ;
\node[font=\tiny,align=center] at (#1,0.5) {$\iiE{#2}$};
\draw[black, fill=black] (#1,0) circle [radius=1.5pt];
\end{scope}
}

\newcommand{\twoed}[2]{
\begin{scope}
\clip (#1-0.2,0.2) rectangle (#1+0.2,-0.6);
\draw[line width=0.5pt, draw=black]  (#1,0) .. controls (#1-0.6,-0.5) and (#1+0.6,-0.5) .. (#1,0) ;
\node[font=\tiny,align=center] at (#1,-0.5) {$\iiE{#2}$};
\draw[black, fill=black] (#1,0) circle [radius=1.5pt];
\end{scope}
}

\newcommand{\starABg}[2]{
\begin{scope}
\clip (#1-0.2,0.2) rectangle (#2+0.2,-0.6);
\draw[line width=0.5pt, draw=lightgray]  (#1,0) .. controls (#1*0.5+#2*0.5-0.6,-0.5) and (#1*0.5+#2*0.5+0.6,-0.5) .. (#2,0) ;
\draw[line width=0.5pt, draw=lightgray]  (#1*0.5+#2*0.5,0) -- (#1*0.5+#2*0.5,-0.38) ;

\draw[line width=0.5pt, draw=lightgray,fill=lightgray]  (#1*0.5+#2*0.5,-0.38) .. controls  (#1*0.15+#2*0.85,-0.37) .. (#2,0) .. controls  (#1*0.25+#2*0.75,-0.06).. (#1*0.5+#2*0.5,-0.38);

\draw[black, fill=black] (#1,0) circle [radius=1.5pt];
\draw[black, fill=black] (#2,0) circle [radius=1.5pt];
\end{scope}
}

\newcommand{\threABg}[2]{
\begin{scope}
\clip (#1-0.2,0.2) rectangle (#2+0.2,-0.6);
\draw[line width=0.5pt, draw=lightgray]  (#1,0) .. controls (#1*0.5+#2*0.5-0.6,-0.5) and (#1*0.5+#2*0.5+0.6,-0.5) .. (#2,0) ;
\draw[line width=0.5pt, draw=lightgray]  (#1*0.5+#2*0.5,0) -- (#1*0.5+#2*0.5,-0.38) ;

\draw[black, fill=black] (#1,0) circle [radius=1.5pt];
\draw[black, fill=black] (#2,0) circle [radius=1.5pt];
\end{scope}
}

\newcommand{\twoeS}[4]{
\begin{scope}
\clip (#1-0.2,1) rectangle (#2+0.2,-0.6);
\draw[decorate, decoration={coil,aspect=0,segment length=2pt,amplitude=1pt}, line width=0.5pt, draw=black]  (#1,0) .. controls (#1*0.5+#2*0.5-#4,#4) and (#1*0.5+#2*0.5+#4,#4) .. (#2,0) node[font=\tiny, midway, above=2pt] {$\iiE{#3}$};
\draw[black, fill=black] (#1,0) circle [radius=1.5pt];
\draw[black, fill=black] (#2,0) circle [radius=1.5pt];

\end{scope}
}

\newcommand{\twoeSd}[4]{
\begin{scope}
\clip (#1-0.2,0.6) rectangle (#2+0.2,-0.6);
\draw[decorate, decoration={coil,aspect=0,segment length=2pt,amplitude=1pt}, line width=0.5pt, draw=black]  (#1,0) .. controls (#1*0.5+#2*0.5-#4,-#4) and (#1*0.5+#2*0.5+#4,-#4) .. (#2,0) node[font=\tiny,midway, below=2.5pt]  {$\iiE{#3}$};
\draw[black, fill=black] (#1,0) circle [radius=1.5pt];
\draw[black, fill=black] (#2,0) circle [radius=1.5pt];
\end{scope}
}

\newcommand{\twoeABS}[3]{
\draw[line width=0.5pt, draw=black]  (#1,0) -- (#2,0) ;
\node[font=\tiny,align=center] at (#1*0.5+#2*0.5,0.1) {$\iiE{#3}$};
\draw[black, fill=black] (#1,0) circle [radius=1.5pt];
\draw[black, fill=black] (#2,0) circle [radius=1.5pt];
}

\newcommand{\twoeARB}[6]{
\draw[line width=0.5pt, draw=black]  (#1,#2) -- (#3,#4) node[font=\tiny, midway, #6] {$\iiE{#5}$} ;
\draw[black, fill=black] (#1,#2) circle [radius=1.5pt];
\draw[black, fill=black] (#3,#4) circle [radius=1.5pt];
}

\newcommand{\thre}[2]{
\begin{scope}
\clip (#1-0.2,-0.2) rectangle (#1+0.2,0.6);
\draw[line width=0.5pt, draw=black] (#1,0)-- (#1,0.38);
\draw[line width=0.5pt, draw=black]  (#1,0) .. controls (#1-0.6,0.5) and (#1+0.6,0.5) .. (#1,0) ;
\node[font=\tiny,align=center] at (#1,0.5) {$\iiiE{#2}$};
\draw[black, fill=black] (#1,0) circle [radius=1.5pt];
\end{scope}
}

\newcommand{\thregs}[2]{
\begin{scope}
\clip (#1-0.2,-0.2) rectangle (#1+0.2,0.6);
\draw[line width=0.5pt, draw=lightgray]  (#1,0) .. controls (#1-#2,0.5) and (#1+#2,0.5) .. (#1,0) ;
\draw[black, fill=black] (#1,0) circle [radius=1.5pt];
\end{scope}
}

\newcommand{\thred}[2]{
\begin{scope}
\clip (#1-0.2,#2+0.2) rectangle (#1+0.2,#2-0.5);
\draw[line width=0.5pt, draw=black] (#1,#2)-- (#1,#2-0.38);
\draw[line width=0.5pt, draw=black]  (#1,#2) .. controls (#1-0.6,#2-0.5) and (#1+0.6,#2-0.5) .. (#1,#2) ;
\draw[black, fill=black] (#1,#2) circle [radius=1.5pt];
\end{scope}
}

\newcommand{\threu}[2]{
\begin{scope}
\clip (#1-0.2,#2-0.2) rectangle (#1+0.2,#2+0.5);
\draw[line width=0.5pt, draw=black] (#1,#2)-- (#1,#2+0.38);
\draw[line width=0.5pt, draw=black]  (#1,#2) .. controls (#1-0.6,#2+0.5) and (#1+0.6,#2+0.5) .. (#1,#2) ;
\draw[black, fill=black] (#1,#2) circle [radius=1.5pt];
\end{scope}
}

\newcommand{\thretr}[3]{
\begin{scope}
\clip (#1-0.2,#2-0.2) rectangle (#1+0.6,#2+0.2);
\draw[line width=0.5pt, draw=black]  (#1,#2) .. controls (#1+0.5,#2-0.6) and (#1+0.5,#2+0.6) .. (#1,#2) ;
\node[font=\tiny,align=center] at (#1+0.5,#2) {$\iiiE{#3}$};
\draw[black, fill=black] (#1,#2) circle [radius=1.5pt];
\draw[line width=0.5pt, draw=black] (#1,#2)-- (#1+0.38,#2);
\end{scope}
}

\newcommand{\thretl}[3]{
\begin{scope}
\clip (#1-0.6,#2-0.2) rectangle (#1+0.2,#2+0.2);
\draw[line width=0.5pt, draw=black]  (#1,#2) .. controls (#1-0.5,#2-0.6) and (#1-0.5,#2+0.6) .. (#1,#2) ;
\node[font=\tiny,align=center] at (#1-0.5,#2) {$\iiiE{#3}$};
\draw[black, fill=black] (#1,#2) circle [radius=1.5pt];
\draw[line width=0.5pt, draw=black] (#1,#2)-- (#1-0.38,#2);
\end{scope}
}

\newcommand{\threAB}[3]{
\begin{scope}
\clip (#1-0.2,0.2) rectangle (#2+0.2,-0.6);
\draw[line width=0.5pt, draw=black] (#2,0)-- (#2,-0.37);
\draw[line width=0.5pt, draw=black]  (#1,0) .. controls (#2-0.8,-0.5) and (#2+0.6,-0.5) .. (#2,0) ;
\node[font=\tiny,align=center] at (#2,-0.5) {$\iiiE{#3}$};
\draw[black, fill=black] (#1,0) circle [radius=1.5pt];
\draw[black, fill=black] (#2,0) circle [radius=1.5pt];
\end{scope}
}

\newcommand{\threABC}[7]{
\draw[line width=0.5pt, draw=black] (#1,#2)-- (#1*0.3333+#3*0.3333+#5*0.3333,#2*0.3333+#4*0.3333+#6*0.3333);
\draw[line width=0.5pt, draw=black] (#3,#4)-- (#1*0.3333+#3*0.3333+#5*0.3333,#2*0.3333+#4*0.3333+#6*0.3333);
\draw[line width=0.5pt, draw=black] (#5,#6)-- (#1*0.3333+#3*0.3333+#5*0.3333,#2*0.3333+#4*0.3333+#6*0.3333)  	;
\draw[black, fill=black] (#1,#2) circle [radius=1.5pt];
\draw[black, fill=black] (#3,#4) circle [radius=1.5pt];
\draw[black, fill=black] (#5,#6) circle [radius=1.5pt];
}

\newcommand{\threABARB}[7]{
\draw[line width=0.5pt, draw=black] (#1,#2)--(#3,#4)  node[font=\tiny, midway, above=3pt] {$\iiiE{#5}$};
\draw[line width=0.5pt, draw=black]  (#1*0.5+#3*0.5,#2*0.5+#4*0.5).. controls (#1*0.5+#3*0.5+#6,#2*0.5+#4*0.5+#7)..(#3,#4);
\draw[black, fill=black] (#1,#2) circle [radius=1.5pt];
\draw[black, fill=black] (#3,#4) circle [radius=1.5pt];
}

\newcommand{\threABS}[3]{
\begin{scope}
\clip (#1-0.2,0.5) rectangle (#2+0.2,-0.5);
\draw[line width=0.5pt, draw=black] (#1,0)-- (#2,0);
\draw[line width=0.5pt, draw=black]  (#1*0.5+#2*0.5,0) .. controls (#1*0.25+#2*0.75-0.1,-0.15) and (#1*0.25+#2*0.75+0.1,-0.15) .. (#2,0);
\node[font=\tiny,align=center] at (#1*0.5+#2*0.5,0.1) {$\iiiE{#3}$};
\draw[black, fill=black] (#1,0) circle [radius=1.5pt];
\draw[black, fill=black] (#2,0) circle [radius=1.5pt];
\end{scope}
}

\newcommand{\threBA}[3]{
\begin{scope}
\clip (#1-0.2,0.2) rectangle (#2+0.2,-0.6);
\draw[line width=0.5pt, draw=black] (#1,0)-- (#1,-0.38);
\draw[line width=0.5pt, draw=black]  (#1,0) .. controls (#1-0.6,-0.5) and (#1+0.6,-0.5) .. (#2,0) ;
\node[font=\tiny,align=center] at (#1,-0.5) {$\iiiE{#3}$};
\draw[black, fill=black] (#1,0) circle [radius=1.5pt];
\draw[black, fill=black] (#2,0) circle [radius=1.5pt];
\end{scope}
}

\newcommand{\threBAu}[3]{
\begin{scope}
\clip (#1-0.2,-0.2) rectangle (#2+0.2,0.6);
\draw[line width=0.5pt, draw=black] (#1,0)-- (#1,0.37);
\draw[line width=0.5pt, draw=black]  (#1,0) .. controls (#1-0.6,0.5) and (#1+0.8,0.5) .. (#2,0) ;
\node[font=\tiny,align=center] at (#1,0.5) {$\iiiE{#3}$};
\draw[black, fill=black] (#1,0) circle [radius=1.5pt];
\draw[black, fill=black] (#2,0) circle [radius=1.5pt];
\end{scope}
}

\newcommand{\foureu}[3]{
\begin{scope}
\clip (#1-0.2,#2-0.2) rectangle (#1+0.2,#2+0.6);
\draw[line width=0.5pt, draw=black]  (#1,#2) .. controls (#1-0.6,#2+0.5) and (#1+0.6,#2+0.5) .. (#1,#2) node[font=\tiny, midway, above=1pt] {$\ivE{#3}$};
\draw[line width=0.5pt, draw=black]  (#1,#2) .. controls (#1-0.3,#2+0.5) and (#1+0.3,#2+0.5) .. (#1,#2);
\draw[black, fill=black] (#1,#2) circle [radius=1.5pt];
\end{scope}
}

\newcommand{\sixeu}[3]{
\begin{scope}
\clip (#1-0.2,#2-0.2) rectangle (#1+0.2,#2+0.6);
\draw[line width=0.5pt, draw=black]  (#1,#2) .. controls (#1-0.6,#2+0.5) and (#1+0.6,#2+0.5) .. (#1,#2) node[font=\tiny, midway, above=1pt] {$\ivE{#3}$};
\draw[line width=0.5pt, draw=black]  (#1,#2) .. controls (#1-0.4,#2+0.5) and (#1+0.4,#2+0.5) .. (#1,#2);
\draw[line width=0.5pt, draw=black]  (#1,#2) .. controls (#1-0.2,#2+0.5) and (#1+0.2,#2+0.5) .. (#1,#2);
\draw[black, fill=black] (#1,#2) circle [radius=1.5pt];
\end{scope}
}

\newcommand{\staru}[3]{
\begin{scope}
\clip (#1-0.2,#2-0.2) rectangle (#1+0.2,#2+0.6);
\draw[line width=0.5pt, draw=black,fill=black]  (#1,#2) .. controls (#1-0.6,#2+0.5) and (#1+0.6,#2+0.5) .. (#1,#2);
\node[font=\tiny,align=center] at (#1,#2+0.5) {$\ivE{#3}$};
\draw[black, fill=black] (#1,#2) circle [radius=1.5pt];
\end{scope}
}

\newcommand{\starug}[2]{
\begin{scope}
\clip (#1-0.2,#2-0.2) rectangle (#1+0.2,#2+0.6);
\draw[line width=0.5pt, draw=lightgray,fill=lightgray]  (#1,#2) .. controls (#1-0.6,#2+0.5) and (#1+0.6,#2+0.5) .. (#1,#2) ;
\draw[black, fill=black] (#1,#2) circle [radius=1.5pt];
\end{scope}
}

\newcommand{\starugd}[3]{
\begin{scope}
\clip (#1-0.2,#2+0.2) rectangle (#1+0.2,#2-0.6);
\draw[line width=0.5pt, draw=lightgray,fill=lightgray]  (#1,#2) .. controls (#1-#3,#2-0.5) and (#1+#3,#2-0.5) .. (#1,#2) ;
\draw[black, fill=black] (#1,#2) circle [radius=1.5pt];
\end{scope}
}

\newcommand{\gS}[1]{
\begin{tikzpicture}[baseline=(A),outer sep=0pt,inner sep=-3pt]
\node (A) at (0,0) {$#1$};
\end{tikzpicture}
}

\def\SL{{\rm SL}}
\def\GL{{\rm GL}}

\def\RR{{\mathbb R}}
\def\PP{{\mathbb P}}
\def\CC{{\mathbb C}}
\def\ZZ{{\mathbb Z}}
\def\NN{{\mathbb N}}

\title[Representations of $\SL_n$ with complete intersection invariant ring]{Completing the classification of representations of $\SL_n$ with complete intersection invariant ring}

\author[Lukas Braun]{Lukas Braun}

\address{Mathematisches Institut, Universit\"at T\"ubingen,
Auf der Morgenstelle 10, 72076 T\"ubingen, Germany}
\email{braun@math.uni-tuebingen.de}


\begin{document}

\begin{abstract}
We present a full list of all representations of the special linear group $\SL_n$ over the complex numbers  with complete intersection invariant ring, completing the classification of Shmelkin. For this task, we combine three techniques. Firstly,  the graph method for invariants of $\SL_n$ developed by the author to compute invariants, covariants and explicit forms of syzygies. Secondly, a new  algorithm for finding a monomial order such that a certain basis of an ideal is a Gr\"obner basis with respect to this order, inbetween usual Gr\"obner basis computation and computation of the Gr\"obner fan. Lastly,   a modification of an algorithm by Xin for MacMahon partition analysis to compute Hilbert series.
\end{abstract}

\maketitle

\tableofcontents

\section*{Introduction}

Let $A=\CC[x_1,\ldots,x_n]/I$ be an affine algebra. We say that $A$ is regular if it is isomorphic to some polynomial ring, a hypersurface if $I$ is a principal ideal, a \emph{complete intersection} if   $I$ is minimally generated by $\mathrm{hd}(A)=n-\mathrm{dim}(A)$ elements, where $\mathrm{dim}(A)$  is the Krull and $\mathrm{hd}(A)$ the homological dimension of $A$.

Now let $\phi: \SL_n \to \mathrm{SL}(W)$ be a finite-dimensional representation of the special linear group $\SL_n$. We denote by $\CC[W]^{\SL_n}$ the \emph{ring of invariants} of this representation. It is a long-standing task in invariant theory to determine these representations that have a regular, hypersurface (hs.) or complete intersection (ci.) ring of invariants respectively.  Representations of connected simple groups with regular ring of invariants have been classified in~\cite{kacpopvin, adgol, schwarzrepreg}, while \emph{irreducible} representations with ci. invariant rings can be found in~\cite{nakaji}. \emph{Reducible} representations of $\SL_2$ with ci. invariant ring  are classically known, see~\cite{gryo}, while those for arbitrary  $\SL_n$ have been classified by Shmelkin in~\cite{shmel} - all but six cases, for which it is still open if they are ci. or not.

The task of the present paper is to address these remaining cases, three single ones and three series to be exact. Let $V$ always denote the standard representation of $\SL_n$ and  by $\BigWedge^k$, $S^k$, $\mathrm{Ad}$ and $S_\lambda$ denote the $k$-th exterior power, $k$-th symmetric power, adjoint representation and Schur module for a vector $\lambda$ of natural numbers respectively. We denote reducible representations by sums of these symbols and the dual by a starred version respectively. The six left open cases are:
$$
(\SL_5, 2\BigWedge^2+\BigWedge^3+V^*),
\quad
(\SL_5, 3\BigWedge^2+V^*),
\quad
(\SL_7, 3V+\BigWedge^3+V^*),
$$
$$
(\SL_{n}, 2\BigWedge^2+4V^*),
\quad
(\SL_{n}, V+2\BigWedge^2+3V^*),
\quad
(\SL_{n}, S^2+\BigWedge^2+2V^*),
\quad
5 \leq n   \in 2\ZZ+1.
$$
We show in Sections~\ref{sec:single} and~\ref{sec:serial} that all six representations have complete intersection invariant rings. Combining this with the results of Shmelkin~\cite[Thms. 0.1, 0.2]{shmel}, we get the following compendium of ci. representations.

\begin{introthm}
A representation of $\SL_n$ has a non-regular but complete intersection invariant ring if and only if it or its dual is contained in the following table. If the invariant ring is not a hypersurface, we note its homological dimension by a boldface subscript.
\begin{center}
\renewcommand{\arraystretch}{1.1}
\begin{tabular*}{\textwidth}{c@{\extracolsep{\fill}}c@{\extracolsep{\fill}}c@{\extracolsep{\fill}}c@{\extracolsep{\fill}}c@{\extracolsep{\fill}}c@{\extracolsep{\fill}}c}
\toprule[0.1em]
$\SL_2$ & $\SL_3$ & $\SL_4$ & $\SL_5$ & $\SL_6$ & $\SL_7$  \\
\midrule
$4V$ & $2S^2\myplus (S^2)^*$ & $S^3$ & $V\myplus 3\BigWedge^2$ & $V\myplus V^*\myplus \BigWedge^2\myplus \BigWedge^3$ & $3V\myplus \BigWedge^3$ &
\\ \addlinespace[1ex]
$3S^2$ & $S^2 \myplus \mathrm{Ad}$ & $6\BigWedge^2$ & $ 3V^*\myplus 2\BigWedge^2$ & $2V^*\myplus \BigWedge^2\myplus \BigWedge^3$ & $2V\myplus 2V^*\myplus \BigWedge^3$ &
\\ \addlinespace[1ex]
$V\myplus 2S^2$ & $V\myplus S^3$ & $V\myplus 4\BigWedge^2$ & $V\myplus 2\BigWedge^2\myplus \BigWedge^3$ & ${2V\myplus \BigWedge^2\myplus \BigWedge^3}_{\iiE{2}}$ & $V\myplus 3V^*\myplus \BigWedge^3$ &
\\ \addlinespace[1ex]
$2V\myplus S^2$ & $V^*\myplus S^3$ & $2V\myplus 3\BigWedge^2$ & ${V^* \myplus  3\BigWedge^2}_{\iiE{2}}$ & ${3V\myplus 3V^*\myplus \BigWedge^3}_{\iiE{2}}$ & $4V^*\myplus \BigWedge^3$ &
\\ \addlinespace[1ex]
$V\myplus S^3$ & $3S^2$ & $V\myplus V^*\myplus 3\BigWedge^2$ & ${V^* \myplus  2\BigWedge^2\myplus \BigWedge^3}_{\iiE{3}}$ & ${4V\myplus V^*\myplus \BigWedge^3}_{\iiE{2}}$ & $\BigWedge^2 \myplus \BigWedge^3$ &
\\ \addlinespace[1ex]
$S^2\myplus S^3$ & $2\mathrm{Ad}$ & $V\myplus S_{(2,2)}$ &
${4V^*+2\BigWedge^2}_{\iiE{4}}$
&  ${4V\myplus 2V^*\myplus \BigWedge^3}_{\iiE{4}}$ & $\BigWedge^3 \myplus \BigWedge^5$ &

\\ \addlinespace[1ex]
$V\myplus S^4$ &&         $\BigWedge^2\myplus S_{(2,2)}$ & & & ${3V\myplus V^*\myplus \BigWedge^3}_{\iiE{2}}$ &
\\ \addlinespace[1ex]
$S^2\myplus S^4$ && ${3\BigWedge^2\myplus S^2}_{\iiE{2}}$ &&& 
\\\cmidrule{6-6}
$2S^4$ &&    ${2\BigWedge^2\myplus \mathrm{Ad}}_{\iiE{3}}$ &&&
$\SL_8$
\\\cmidrule{6-6}
$S^5$ &&
${2V \myplus 2V^* \myplus 2\BigWedge^2}_{\iiE{2}}$
&&&
$2V \myplus \BigWedge^3$ 
\\ \addlinespace[1ex]

$S^6$ &&
${3V \myplus V^* \myplus 2\BigWedge^2}_{\iiE{2}}$
&&&
$V \myplus V^* \myplus \BigWedge^3$
\\ \addlinespace[1ex]
${2S^3}_{\iiE{2}}$ &&
${4V   \myplus 2\BigWedge^2}_{\iiE{3}}$
&&&
$2V^* \myplus \BigWedge^3$
\\ \addlinespace[1ex]
\toprule[0.1em]
\end{tabular*}
\begin{tabular*}{\textwidth}{c@{\extracolsep{\fill}}c@{\extracolsep{\fill}}c}
 \multicolumn{3}{c}{$\SL_n$}
 \\
 \midrule
$n\geq 3$ & $n \geq 4 $ & $n \geq 5$
 \\
\midrule
$nV \myplus nV^*$
&
$nV^* \myplus \BigWedge^2, n \mathrm{~even}$
&
$
{kV \myplus lV^* \myplus 2\BigWedge^2}_{\iiE{\mathrm{max}(\lfloor \frac{n}{2} \rfloor,n-l-1)}}, 
$
\\ \addlinespace[1ex]
$nV^* \myplus S^2$
&
$V \myplus nV^* \myplus \BigWedge^2$
&
$k+l=4, l \leq n-2.$
\\ \addlinespace[1ex]
$V \myplus (n \mynus 1)V^* \myplus S^2$
&
$3V \myplus (n \mynus 1)V^* \myplus \BigWedge^2$
&
${4V+\BigWedge^2 \myplus (\BigWedge^2)^*}_{\iiE{\frac{n+2}{2}}}, n \mathrm{~even}$
\\ \addlinespace[1ex]
$2V \myplus rV^*  \myplus S^2, r\leq n \mynus 3$
&
$4V \myplus rV^* \myplus \BigWedge^2, r \leq n \mynus 3$
&
${3V \myplus V^* \myplus \BigWedge^2 \myplus (\BigWedge^2)^*}_{\iiE{\lceil\frac{n}{2}}\rceil}$
\\ \addlinespace[1ex]
$V \myplus 2S^2$
&
${2V \myplus nV^* \myplus \BigWedge^2}_{\iiE{2}}$
&
${2V \myplus 2V^* \myplus \BigWedge^2 \myplus (\BigWedge^2)^*}_{\iiE{\lfloor\frac{n}{2}}\rfloor}$
\\ \addlinespace[1ex]
$V^* \myplus 2S^2$
&
${4V \myplus (n \mynus 2)V^* \myplus \BigWedge^2}_{\iiE{2}}$
&
$
{2V \myplus  S^2 \myplus (\BigWedge^2)^*}_{\iiE{\frac{n+2}{2}}}, n \mathrm{~even}
$
\\ \addlinespace[1ex]
 $V \myplus S^2 \myplus (S^2)^*$
 &
 $
 {2V \myplus \BigWedge^2 \myplus S^2}_{\iiE{n \mynus 1}}
 $
 &
 $
{ V \myplus V^* \myplus S^2 \myplus (\BigWedge^2)^*}_{\iiE{\lceil\frac{n}{2}\rceil}}
$
 \\ \addlinespace[1ex]
  $V \myplus V^* \myplus \mathrm{Ad}$
  &
   $
 {V \myplus V^* \myplus \BigWedge^2 \myplus S^2}_{\iiE{n \mynus 2}}
 $
 &
  $
{ 2V^* \myplus S^2 \myplus (\BigWedge^2)^*}_{\iiE{\lfloor\frac{n}{2}\rfloor}}
$
  \\ \addlinespace[1ex]
  ${2V \myplus (n \mynus 2) V^* + S^2}_{\iiE{2}}$
  &
    $
 {2V^* \myplus \BigWedge^2 \myplus S^2}_{\iiE{\mathrm{max}(2,n \mynus 3)}}
 $
 &
 $V \myplus S^2 \myplus (\BigWedge^2)^*$
  \\ \addlinespace[1ex]
  ${(n \myplus 1)V \myplus r V^*}_{\iiE{r}}, r \leq n$
  &
  &
   $3V \myplus \BigWedge^2 \myplus (\BigWedge^2)^*$
   \\ \addlinespace[1ex]
   \bottomrule[0.1em]
\end{tabular*}
\end{center}
\end{introthm}

For most non-serial cases from~\cite{shmel}, generators and at least multidegrees of syzygies are given there. Generators and syzygies for the six left open cases can be found in Propositions~\ref{prop:U1}-\ref{prop:U3},~\ref{prop:W1},~\ref{prop:W2},~\ref{prop:W3} of the present work and the respective proofs. With the graph techniques from these proofs, it is moreover possible to determine explicit forms of syzygies for the serial cases from~\cite{shmel}.

The setup of the paper is as follows. In Section~\ref{sec:tech}, we discuss the three main techniques involved, which are the graph approach for invariants of $\SL_n$ developed by the author in~\cite{braun}, a new algorithm for determining monomial orderings for a certain \emph{desired} Gr\"obner basis (with \emph{desired} leading terms)  of an ideal, and a slight modification of an algorithm by Xin~\cite{Xin} for MacMahon partition analysis.
We want to stress that all three techniques are applicable in greater generality and that the present work may serve as a blueprint for such applications.

Section~\ref{sec:single} covers the three single cases while the main focus lies on the serial cases in Section~\ref{sec:serial}, where we compute the invariant rings in terms of generators and syzygies. This is done with the help of covariant rings of subrepresentations and toric invariant rings thereof. The resulting algebras $A$ are 'almost' toric - the ideal of syzygies being generated by binomials as well as trinomials, such that a torus  of dimension $\mathrm{dim}(A)-1$ acts on them, compare~\cite{ltpticr}. These almost toric algebras are contractions (i.e. special fibers of degenerations) of the original invariant rings and thus provide information on multidegrees of generators and syzygies. See~\cite{HMSV} for a related approach that uses both a graphical method for the invariants and a - toric - degeneration of the invariant ring. 

\subsection*{Acknowledgements}
The author wishes to thank Dmitri Shmelkin for pointing out his classification to him.

\section{The basic techniques}\label{sec:tech}

In the following, three different techniques are presented: the graph approach for $\SL_n$-invariants from~\cite{braun}, a Gr\"obner basis algorithm for finding suitable monomial orderings and a modification of an algorithm for constant term evaluation of rational functions, see~\cite{Xin} for the original algorithm.
In the first case, the work~\cite{braun} is exhaustive in the sense that it discusses techniques for finding generators and syzygies for $\SL_n$-invariant rings of antisymmetric tensors. What is new in the present work is that we include symmetric tensors in one case and also \emph{covariants}. Another application of the graph method to covariants in the case of $\SL_4$ can be found in the work~\cite{sl4cov} by the author. 

\subsection{The graph method for (anti-)symmetric tensors}

The following is merely a summary of~\cite[Sec. 2, 3]{braun}, while the basis for the methods therein are the (skew) brackets from~\cite{GRS}.
Both should be considered for a deeper discussion. 
Let $V$ denote the standard representation of $\SL_n$ with standard basis $(e_i)$ and dual basis $(e_i^*)$. For nonnegative $i,j, \iota, \gamma$ we denote $\BigWedge_{i,j}:=\BigWedge^i (V)$ and  $S_{\iota,\gamma}:=S^\iota (V)$. Let also $n_i, m_\iota$ be nonnegative, then we have a representation
$$
W:=\bigoplus_i \left( \bigoplus_{j=1}^{n_i} \BigWedge_{i,j} \right) \oplus \bigoplus_\iota \left( \bigoplus_{\gamma=1}^{m_\iota} S_{\iota,\gamma} \right).
$$
To a subspace $\BigWedge_{i,j}$, $S_{\iota,\gamma}$ of $W$ associate an infinite supply of letters $a_{ijk}$, $b_{\iota \gamma \kappa}$ respectively. A \emph{bracket}, denoted by $[*]$, contains $n$ of these letters. We consider the free algebra generated by these brackets, with respect to the following relations, and denote the resulting algebra by $\mathrm{Bra}$.
\begin{enumerate}
\item Inside a bracket, the letters behave commutatively for all combinations but two letters of type $b$, which behave anticommutatively, i.e.
$$
[*a_{ijk}a_{i'j'k'}*]=[*a_{i'j'k'}a_{ijk}*],
\quad [*a_{ijk}b_{\iota \gamma \kappa}*]=[*b_{\iota \gamma \kappa}a_{ijk}*],
$$
$$ [*b_{\iota \gamma \kappa}b_{\iota' \gamma' \kappa'}*]=-[*b_{\iota' \gamma' \kappa'}b_{\iota \gamma \kappa}*].
$$
\item If $n$ is even, $\mathrm{Bra}$ is commutative in the brackets, while it is anticommutative for odd $n$.
\item A monomial in brackets is nonzero only if a letter $a_{ijk}$ ($b_{\iota \gamma \kappa}$) appears either zero or exactly $i$ ($\iota$) times inside the monomial.
\item The Pl\"ucker relations (or \emph{Exchange Lemma}~\cite[p. 60]{GRS}):
$$
\sum_{(u',u'')\vdash u} [u'v][u''w] = (-1)^{n-|w|} \sum_{(v',v'')\vdash v} [v' u] [v'' w],
$$
where the notation $(u',u'')\vdash u$ means that we sum over all decompositions  of $u$ in two subwords $u',u''$ and all occuring summands where a bracket contains more than $n$ letters are set to zero. If $u$ or $v$ contain letters of type $b$, we get different presigns - see~\cite[Prop. 10]{GRS} for the statement in whole generality - but we omit this here for simplicity as it is not needed.
\end{enumerate} 

Observe that the above Pl\"ucker relation does not include signs of the underlying permutations as the standard Pl\"ucker relations do. This is due to the construction in~\cite{GRS}. The signs will appear when we assign polynomials in  $\CC[W]^{\SL_n}$ to the bracket polynomials.

In order to do so, fix a total order on the letters $a_{ijk}$, $b_{\iota \gamma \kappa}$. We have a surjective linear map $U$ from $\mathrm{Bra}$ to $\CC[W]^{\SL_n}$ - called the umbral operator in~\cite{GRS} - defined by linearly extending the following definition for bracket monomials to polynomials.
The image $U(\mathfrak{m})$ of a  bracket monomial $\mathfrak{m}$ is the polynomial mapping an element $\sum t_{i,j} + \sum s_{\iota,\gamma}$ to the following: 
\begin{enumerate}
\item Inside each bracket, arrange letters with respect to the chosen total order.
\item Take the tensor product of one tensor  $\mathrm{det}:=e_1^* \wedge \ldots \wedge e_n^*$ for each bracket and one tensor $t_{i,j}$ ($s_{\iota,\gamma}$) for each letter $a_{ijk}$ ($b_{\iota \gamma \kappa}$) appearing in the monomial.
\item The image under $U(\mathfrak{m})$ is the complete contraction, where the $\nu$-th index of  $t_{i,j}$ ($s_{\iota,\gamma}$) and $\mu$-th index of some $\mathrm{det}$ are contracted if and only if the $\nu$-th appearance of the corresponding $a_{ijk}$ ($b_{\iota \gamma \kappa}$) is at the $\mu$-th  position of the respective bracket.
\end{enumerate}

In~\cite{shmel}, tensor contractions are used to compute invariants without utilizing brackets. Compare also~\cite[Sec. 9.5]{PV}. 
The content of the graph method from~\cite{braun} is  to use an algebra of hypergraphs instead of the bracket algebra $\mathrm{Bra}$ by identifying a bracket monomial with a graph that has a vertex for each bracket and an $i$-hyperedge of \emph{color} $j$ and \emph{shading} $k$ for each letter $a_{ijk}$ that is connected to a vertex if and only if the letter turns up in the corresponding bracket. As the shading only affects the presign, we often omit it and denote only the color of hyperedges. We will also omit the 'hyper' and only speak of graphs and $k$-edges. Moreover, we will often refuse to distinguish between the graph and the associated invariant when the meaning is clear.
A somehow dual approach for invariants of binary forms can be found in~\cite{olvgraph}.

Now we extend the graph method on the one hand to brackets containing letters $b_{\iota \gamma \kappa}$ by assigning jagged $\iota$-edges to such letters. On the other hand, we extend to covariants by adding gray dummy $k$-edges (behaving as if they correspond to tensors $e_1\wedge \ldots \wedge e_k$), compare~\cite{sl4cov, GRS, RS, shmel3} and Subsection~\ref{subs:cov} of Section~\ref{sec:serial}.

Now the Pl\"ucker relation from above becomes a graph relation in the algebra generated by these graphs. We will from now on say that we \emph{apply the Pl\"ucker relation to some edges} if these edges correspond to the word $u$ in the original Pl\"ucker relation for brackets. If it is not clear to which vertex the bracket with the word $v$ corresponds, we say we \emph{pull the edges to the respective vertex}.
A disconnected graph is the  product of its connected components, we say that a graph or a sum of graphs is \emph{reducible}, if it equals a sum of disconnected graphs. We say that two irreducible graphs are \emph{reducibly equivalent}, if their difference is reducible. A set of irreducible graphs if said to be \emph{reducibly independent}, if the only linear combination of them that is reducible is the trivial one. A minimal generating set of the ring of invariants $\CC[W]^{\SL_n}$ is in one-to-one-correspondence to a maximal set of reducibly independent irreducible graphs, see~\cite[Thm. 3.8]{braun}.

The advantage of  the graph method is twofold: firstly, one can use graph theoretical methods and secondly, it is somehow easier to see the consequences of Pl\"ucker relations or how they can be applied reasonably. The following example gives a blueprint of how these advantages interact.

\begin{example}
Let $n=4$ and $n_i,m_\iota=0$ except for $n_2\neq 0$. So we consider invariants of $\SL_4$ acting on antisymmetric two-tensors. The respective graphs have vertices of degree four and only non-jagged $2$-edges (i.e. are truly graphs). We say that an edge is looping if it connects to the same vertex with both ends. 
We want to find a generating set for the associated ring of invariants. Let $\Gamma$ be a graph as above. If $\Gamma$ has a vertex with two looping edges, this constitutes a connected component and corresponds to a generator of $\CC[W]^{\SL_n}$, whatever color the two edges have. So we can assume that $\Gamma$ possesses no such vertex.
If there is a vertex with not even one looping edge, then apply the Pl\"ucker relation on any edge connected with this vertex and pulling the edge to the vertex. This results in $\Gamma$ being equal to a sum of graphs with a looping edge at this very vertex and no looping edges of other vertices being affected. Applying this procedure to every vertex with no looping edge, we can assume that each vertex \emph{has} a looping edge. So by basic graph theory $\Gamma$ equals a sum of cycles of the form
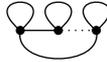
\begin{center}
 \begin{tikzpicture}[baseline=(A),outer sep=0pt,inner sep=0pt]
\node (A) at (0,0) {};
\clip (-0.2,-0.5) rectangle (1.2,0.5);
\draw[line width=0.5pt, draw=black]  (0,0) .. controls (0-0.6,0.5) and (0+0.6,0.5) .. (0,0) ;
\draw[line width=0.5pt, draw=black]  (0.5,0) .. controls (0.5-0.6,0.5) and (0.5+0.6,0.5) .. (0.5,0) ;
\draw[line width=0.5pt, draw=black] (0,0)-- (0.5,0);
\draw[line width=0.5pt, draw=black, dotted] (1,0)-- (0.5,0);
\draw[line width=0.5pt, draw=black]  (0,0) .. controls (0.5-0.6,-0.5) and (0.5+0.6,-0.5) .. (1,0) ;
\draw[line width=0.5pt, draw=black]  (1,0) .. controls (1-0.6,0.5) and (1+0.6,0.5) .. (1,0) ;
\draw[black, fill=black] (0,0) circle [radius=1.5pt];
\draw[black, fill=black] (0.5,0) circle [radius=1.5pt];
\draw[black, fill=black] (1,0) circle [radius=1.5pt];
\end{tikzpicture}
\end{center}

Moreover, we can show by applying the Pl\"ucker relation to a looping edge and a non-looping edge connected to the same vertex, that the negative of the graph $\Gamma'$ that has these two edges interchanged differs from $\Gamma$ by a disconnected graph. Thus when considering a generating set, one can assume that the cycles are antisymmetric in the edges. 

But the generating set consisting of all these cycles is by no means minimal. A minimal generating set only consists of the vertices with two looping edges and the cycles with \emph{three} vertices. This was proven in~\cite[Proof of Prop. 4.1]{braun} by applying Pl\"ucker relations in a way that is somewhat natural in the graph setting and very hard to see working with brackets. Similar techniques are constantly applied in~\cite[Sec. 4-6]{braun}.
\end{example} 

\subsection{The Crosshair-sieve-algorithm}

The Crosshair-sieve-algorithm~\ref{algCS} is developed in Section~\ref{sec:serial} to show that a certain ideal basis is a Gr\"obner basis in fact. We give a general but very rough account of this algorithm here.
It lies  between standard Gr\"obner basis computations with usually a limited amount of available monomial orders and the computation of the Gr\"obner fan, introduced in~\cite{gröbfan}. In particular, if we want to find a Gr\"obner basis with certain good properties (for instance a particular set of leading monomials), then it is in general not very likely that the standard monomial orders such as lexicographic, total degree reverse lexicographic et cetera produce such basis. This problem is adressed by Gr\"obner fan computations, see~\cite{compgröb}. On the other hand, the computation of the whole Gr\"obner fan encoding all possible  Gr\"obner bases of one ideal or a universal - i.e. with respect to all monomial orders - Gr\"obner basis, see~\cite{StuGrö}, can be too complex, see~\cite{beygröb} for such problems in applied Gr\"obner basis theory. This is in particular the case if some parameters are involved as in our three series of invariant rings. Moreover, even the computation of a single Gr\"obner basis for a badly suited monomial order can be too complex in such cases.

So assume we are in the following situation: we have a finite set of polynomials $f_i$ with monomials $m_i$ appearing in the $f_i$ and we want to find a monomial order such that the $f_i$ have leading monomials $\mathrm{lm}(f_i)=m_i$. If for example the $m_i$ pairwise have no variable in common, then the $f_i$ are already a Gr\"obner basis of the ideal generated by themselves. This follows directly from the Buchberger algorithm~\cite[Thm. 3.3]{buchb} as all $S$-polynomials reduce to zero.

Any monomial order on $s$ variables can be expressed as a matrix order, see~\cite{robb, kemp}, with a matrix $\mathcal{M} \in \GL_s(\RR)$ where degrees are given by the first row of $\mathcal{M}$ with ties broken by the second row et cetera. In fact, we only require $\mathcal{M}$ to have $s$ columns, an arbitrary number of rows and rank $s$. What we want to do is building up a matrix $\mathcal{M}$ such that the induced monomial order gives the required leading terms.

\begin{algorithm}[\emph{General Crosshair-sieve}]\label{algCSG}
In order to do so, let $d$ be the maximal total degree of all $f_i$ and denote the variables by $x_1,\ldots,x_s$. Set $x_0:=1$. To each monomial $x_1^{a_1}\cdots x_s^{a_s}$ associate a degree vector $\sum a_j e_j \in \RR^s$. Denote the rows of $\mathcal{M}$ by sums $\mathcal{M}_\nu=\sum b_j e_j^*$. 

 To each row $\mathcal{M}_\nu$ of $\mathcal{M}$ associate a symmetric tensor $S_\nu \in S^d(\RR^{s+1})$, the so called \emph{sieve}, collecting the degrees associated by $\mathcal{M}_\nu$ to all monomials in the $x_i$ of degree $\leq d$. In the entry $(S_\nu)_{i_1,\ldots,i_d}$ stands the degree of $x_{i_1} \cdots x_{i_d}$. 
 
 The degrees of all monomials of one $f_i$ are a collection of entries $(S_\nu)_{i_1,\ldots,i_d}$. Now we 'target' a monomial $m_i$ by increasing the coefficient $b_j$ in  $\mathcal{M}_\nu$ for at least one $x_j$ occuring in $m_i$. We do this for every $m_i$ \emph{in such way that $\mathrm{deg}_\nu(m_i)\geq \mathrm{deg}_\nu(m'_i)$ for all monomials $m'_i$ of $f_i$}.
 It is necessary for the algorithm to work (and for the $m_i$ to be a possible set of leading monomials) that such way exists. This is the case for example if all $f_i$ are homogeneous and we set $b_j=1$ for all $x_j$ occuring in some $m_i$. But this condition is not sufficient of course. At least for one single $i$, we require $\mathrm{deg}_\nu(m_i)$ to be truly greater than $\mathrm{deg}_\nu(m'_i)$ for all other monomials $m'_i$ of $f_i$. We say $m_i$ is \textit{filtered out} by $S_\nu$. Which means we do not have to target this very $m_i$ in $S_{\nu+1}$ any more and thus get less \textit{unwanted interferences of $b_j$} leading to  high degrees of unwanted monomials. It is clear that this algorithm terminates if and only if at each step - for each row of $\mathcal{M}$ - at least one $m_i$ is filtered out.
\end{algorithm}

\begin{remark}
The success of this approach of course relies heavily on the structure of the $f_i$, i.e. how the degrees of their monomials are distributed over $S_\nu$. In practice, for example in invariant theory, the $f_i$ will most likely inherit some symmetry like weighted homogeneity and a symmetric distribution over $S_\nu$. 
\end{remark}

\begin{example}
We give a short account of how such properties fit  together very well in the serial cases considered in Section~\ref{sec:serial}. For details we refer to Algorithm~\ref{algCS}. Due to weighted homogeneity, even if the $m_i$ are not of maximal total degree under all monomials in the $f_i$, there will be some variables occuring in the $m_i$ but not in the $m'_i$ of higher total degree. Targeting these variables at first will filter out monomials of the same total degree as the $m_i$. Now we can arrange the variables in such way that degrees of one $f_i$ lie on counterdiagonals in the $2$-tensors $S_\nu$ (seen as symmetric matrices). Moreover, as we have some freedom in the choice of the $m_i$, we can arrange them on diagonals. Now targeting them by setting $b_j=1$ for each $x_j$ in some $m_i$ that has not yet been filtered out, we see that at least the right- and lowermost entry of $S_\nu$ will be filtered out at each step. The picture of  $S_\nu$ with one $m_i$ targeted  in this way is exactly that of a \emph{crosshair}.
\end{example}

\subsection{Hilbert series by MacMahon partition analysis}

We compute Hilbert series of the three single cases in Section~\ref{sec:single} and of low dimensional representatives of the serial cases in Section~\ref{sec:serial} by the following method. 

According to~\cite[Sec. 4.6]{KD} the (univariate) Hilbert series of the invariant ring of a representation of a reductive algebraic group can be expressed as the \emph{constant term} in $l$ variables $z_1,\ldots,z_l$ of a rational function of the form
$$
\frac{f(z_1,\ldots,z_l)}{(1-f_1(z_1,\ldots,z_l)t)\cdots(1-f_k(z_1,\ldots,z_l)t)}.
$$
It is one of the main objectives of \emph{MacMahon partition analysis} to compute constant terms of so called \emph{Elliot-rational} functions, with various applications and implementations, see~\cite{omega, Xin, Xin2}. Constant terms in more than one variable $z_1,\ldots,z_l$ are treated as iterated constant terms in one variable $z_i$. The constant term can be directly read off if one has a partial fraction decomposition of the rational function. This is the approach of the algorithm $\mathtt{Ell}$ developed in~\cite{Xin} with an implementation in Maple. But there are two bottlenecks: firstly, if one of the denominator factors is not \emph{linear} in $z_i$ and secondly, if two of the denominator factors are not relatively prime. These problems lead to a very significant increase in runtime, which makes it practically impossible to address complicated problems. In theory, the first problem can be solved by introducing roots of unity, the second one by introducing 'slack variables' $s_i$, one for each denominator factor, so that the function will change to
$$
\frac{f(z_1,\ldots,z_l)}{(1-f_1(z_1,\ldots,z_l)ts_1)\cdots(1-f_k(z_1,\ldots,z_l)ts_k)}.
$$
This is the approach of~\cite{Xin2}. But the output of such algorithm is a large sum of rational functions - most of them with poles at some $s_i=1$ - which must be simplified to a single rational function where we can set $s_i=1$. So the problem of computational complexity is only shifted. In~\cite{Xin2}, MacMahon partition analysis then is used again to compute constant terms in $v_i=s_i-1$ for each of the output functions, which may by far be the part of the computation with the highest complexity, see~\cite[Sec. 5.3]{Xin2}. 

In the Hilbert series computations of the present paper, we did not succeed in reasonable time with any of the mentioned algorithms. So we propose the following modified version: assume the only denominator factors that are not relatively prime are equal. Instead of introducing a slack variable for each factor, we introduce \emph{one} slack variable $s$ and take different powers of $s$ for each member of a collection of identical denominator factors. A function of the form
$$
\frac{f}{(1-f_1t)^{n_1}\cdots(1-f_kt)^{n_k}}
$$
will thus become
$$
\frac{f}{(1-f_1t)\cdots(1-f_1ts^{n_1-1})\cdots(1-f_kt)\cdots(1-f_kts^{n_k-1})}.
$$
The output of the algorithm from~\cite{Xin} applied to such function will be a large sum of rational functions with no chance to be simplified by e.g. standard Maple simplification. This is where the second aspect of our modification comes into play: most of the rational functions will have a pole in $s=1$, but of different \emph{order}. Let $c \in \NN$ be the highest pole order. Since we know that the resulting function has no pole at $s=1$, the sum of all functions with a pole of order $c$ must have a pole of order $c-1$ at most. Since it is likely that such sum consists of not nearly as much terms as the whole sum, it might be possible to compute it and thus reduce the highest overall pole order. Iterating this procedure will of course finally result in the sum of all rational functions, our desired output. 

The Maple-worksheets for the Hilbert series computations of the $\SL_5$ representations $2\BigWedge^2+\BigWedge^3+V^*$, $
 3\BigWedge^2+V^*$ and $2\BigWedge^2+4V^*$ will be available as supplementary material. We claim that the described modification is useful for at least problems of certain complexity.
 
 \begin{remark}
 The above algorithms may analogically be used to compute multivariate Hilbert series, i.e. not with respect to total degree but multigradings. As such Hilbert series contain more information, compare Section~\ref{sec:serial}, they may be more useful in some situations than the usual univariate ones, for example if one is searching for generators and relations. 
 But then on the other hand, one has to compute multivariate Hilbert series of ideals given by generators and multidegrees. To our knowledge, there is no such implementation available, though it is analogue to computation of univariate Hilbert series by means of Gr\"obner bases. A basic implementation in Maple, applied to computing the Hilbert series of the algebra $A_3$ from Lemma~\ref{le:W3}, is provided as supplementary material as well.  
 \end{remark}

\section{The single cases}\label{sec:single}

Consider the $\SL_5$-representations 
$
U_1:= 2\BigWedge^2+\BigWedge^3+\BigWedge^4$
,
$U_2:= 3\BigWedge^2+\BigWedge^4$
and the $\SL_7$-representation
$U_3:= 3V+\BigWedge^3+\BigWedge^6$.
These are the three single cases that Shmelkin left open in~\cite{shmel}. We prove that all three are complete intersections in the following.

\subsection{The case $\left(\SL_5, U_1\right)$}

\begin{proposition} \label{prop:U1}
The ring of invariants is a complete intersection of homological dimension two minimally generated by
\begin{center}
\gS{g_{122}=}
\begin{tikzpicture}[baseline=(A),outer sep=0pt,inner sep=0pt]
\node (A) at (0,0) {};
\foureu{0}{0}{1}
\twoeABS{0}{0.7}{1}
\twoe{0.7}{2}
\twoed{0.7}{2}
\end{tikzpicture}
\kern-0.5em
\gS{,}
\
\gS{g_{211}=}
\begin{tikzpicture}[baseline=(A),outer sep=0pt,inner sep=0pt]
\node (A) at (0,0) {};
\foureu{0}{0}{1}
\twoeABS{0}{0.7}{2}
\twoe{0.7}{1}
\twoed{0.7}{1}
\end{tikzpicture}
\kern-0.5em
\gS{,}
\
\gS{g_{11}=}
\begin{tikzpicture}[baseline=(A),outer sep=0pt,inner sep=0pt]
\node (A) at (0,0) {};
\twoe{0}{1}
\twoed{0}{1}
\threABS{0}{0.7}{1}
\thre{0.7}{1};
\end{tikzpicture}
\kern-0.5em
\gS{,}
\
\gS{g_{22}=}
\begin{tikzpicture}[baseline=(A),outer sep=0pt,inner sep=0pt]
\node (A) at (0,0) {};
\twoe{0}{2}
\twoed{0}{2}
\threABS{0}{0.7}{1}
\thre{0.7}{1};
\end{tikzpicture}
\kern-0.5em
\gS{,}
\
\gS{g_{12}=}
\begin{tikzpicture}[baseline=(A),outer sep=0pt,inner sep=0pt]
\node (A) at (0,0) {};
\twoe{0}{1}
\twoed{0}{2}
\threABS{0}{0.7}{1}
\thre{0.7}{1};
\end{tikzpicture}
\kern-0.5em
\gS{,}
\vspace{2pt}
\gS{g=}
\begin{tikzpicture}[baseline=(A),outer sep=0pt,inner sep=0pt]
\node (A) at (0,0) {};
\foureu{0}{0}{1}
\threABS{0}{0.7}{1}
\thre{0.7}{1};
\end{tikzpicture}
\kern-0.5em
\gS{,}
\
\gS{g_{1}=}
\begin{tikzpicture}[baseline=(A),outer sep=0pt,inner sep=0pt]
\node (A) at (0,0) {};
\thre{0}{1}
\twoed{0}{1}
\end{tikzpicture}
\kern-0.5em
\gS{,}
\
\gS{g_{2}=}
\begin{tikzpicture}[baseline=(A),outer sep=0pt,inner sep=0pt]
\node (A) at (0,0) {};
\thre{0}{1}
\twoed{0}{2}
\end{tikzpicture}
\kern-0.5em
\gS{,}
\
\gS{g_{112122}=}
\begin{tikzpicture}[baseline=(A),outer sep=0pt,inner sep=0pt]
\node (A) at (0,0) {};
\twoe{0}{1}
\twoed{0}{1}
\twoeABS{0}{0.7}{2}
\thre{0.7}{1};
\twoeABS{0.7}{1.4}{1}
\twoe{1.4}{2}
\twoed{1.4}{2}
\end{tikzpicture}
\kern-0.5em
\gS{,}
\
\gS{g_{1122}=}
\begin{tikzpicture}[baseline=(A),outer sep=0pt,inner sep=0pt]
\node (A) at (0,0) {};
\foureu{0}{0}{1}
\twoeABS{0}{0.7}{1}
\thre{0.7}{1};
\twoeABS{0.7}{1.4}{1}
\twoe{1.4}{2}
\twoed{1.4}{2}
\end{tikzpicture}
\kern-0.5em
\gS{,}
\\
\gS{g_{2122}=}
\begin{tikzpicture}[baseline=(A),outer sep=0pt,inner sep=0pt]
\node (A) at (0,0) {};
\foureu{0}{0}{1}
\twoeABS{0}{0.7}{2}
\thre{0.7}{1};
\twoeABS{0.7}{1.4}{1}
\twoe{1.4}{2}
\twoed{1.4}{2}
\end{tikzpicture}
\kern-0.5em
\gS{,}
\
\gS{g_{1211}=}
\begin{tikzpicture}[baseline=(A),outer sep=0pt,inner sep=0pt]
\node (A) at (0,0) {};
\foureu{0}{0}{1}
\twoeABS{0}{0.7}{1}
\thre{0.7}{1};
\twoeABS{0.7}{1.4}{2}
\twoe{1.4}{1}
\twoed{1.4}{1}
\end{tikzpicture}
\kern-0.5em
\gS{,}
\
\gS{g_{21}=}
\begin{tikzpicture}[baseline=(A),outer sep=0pt,inner sep=0pt]
\node (A) at (0,0) {};
\foureu{0}{0}{1}
\twoeABS{0}{0.7}{1}
\thre{0.7}{1};
\twoeABS{0.7}{1.4}{2}
\foureu{1.4}{0}{1}
\end{tikzpicture}
\gS{.}
\end{center}
\end{proposition}

\begin{proof}
The list of generators can be extracted from~\cite[Th 1.3]{braun}. Since the Krull dimension of the ring of invariants is $11$, the homological dimension is two and it is a complete intersection due to~\cite[Rem. to Prop. 1.5]{popsyz}. We also computed the Hilbert series of $\CC[U_1]^{\SL_5}$, it is
$$
\frac{(1-t^{10})(1-t^{12})}{(1-t^2)^2(1-t^3)(1-t^4)^5(1-t^5)(1-t^6)^3(1-t^7)}.
$$
Moreover, as can be seen by comparing Hilbert series, generators of the ideal of syzygies are
$$
16g_{1122}g_{2}g_{211} -6g_{2}^2g_{211}^2 - 3g_{1}g_{1122}g_{122} - 8g_{11}g_{122}^2 - 8g_{12}g_{122}g_{211} - 3g_{112122}g_{21}
$$
$$
 - 16g_{1122}^2 - 16g_{1211}g_{2122},
$$
$$
2g g_{112122} - 4g_{12} g_{1122} -2 g_{2122} g_{11} -2 g_{1211} g_{22} -g_{122} g_2 g_{11} +g_{211}g_1 g_{22} - 2 g_{211} g_2 g_{12}.
$$
\end{proof}

\subsection{The case $\left(\SL_5, U_2\right)$}

\begin{proposition}\label{prop:U2}
The ring of invariants is a complete intersection of homological dimension three minimally generated by

\begin{tabular}{rl}
\gS{f_{abcde}=}
\begin{tikzpicture}[baseline=(A),outer sep=0pt,inner sep=0pt]
\node (A) at (0,0) {};
\twoe{0}{a}
\twoed{0}{b}
\twoeABS{0}{0.7}{c}
\twoe{0.7}{d}
\twoed{0.7}{e}
\node (B) at (0,-0.75) {};
\end{tikzpicture}
&
\gS{, abcde \in \{12311,12322,13233,11233,11322,22133\};}
\\[2pt]
\gS{f_{abc}=}
\begin{tikzpicture}[baseline=(A),outer sep=0pt,inner sep=0pt]
\node (A) at (0,0) {};
\foureu{0}{0}{1}
\twoeABS{0}{0.7}{a}
\twoe{0.7}{b}
\twoed{0.7}{c}
\end{tikzpicture}
&
\gS{, abc \in \{122,322,233,133,123,213\}.}
\end{tabular}
\end{proposition}

\begin{proof}
The list of generators can be extracted from~\cite[Th 1.3, Prop. 5.1]{braun}. Since the Krull dimension of the ring of invariants is $11$, the homological dimension is three. We need to show that the ideal of relations is generated by three polynomials. The Hilbert series of $\CC[W]^{\SL_5}$ is
$$
\frac{(1-t^9)^3}{(1-t^4)^8(1-t^5)^6}
$$ 
 and the following three relations hold between the invariants:
$$
-2f_{122}f_{13233}-2f_{322}f_{11233}+f_{233}f_{11322}-f_{133}f_{12322}+2f_{123}f_{22133},
$$
$$
2f_{211}f_{13233}-2f_{311}f_{22133}+f_{133}f_{11322}-f_{233}f_{12311}+2f_{213}f_{11233},
$$
$$
2f_{311}f_{12322}-2f_{211}f_{22133}+f_{122}f_{11233}+f_{322}f_{12311}-2(f_{213}+f_{123})f_{11322}.
$$
The Hilbert series of $\CC[f_{abcde},f_{abc}]/I$ with the ideal $I$ generated by these three polynomials and the above one of $\CC[W]^{\SL_5}$ coincide, thus the assertion follows.

\end{proof}

\subsection{The case $\left(\SL_7,U_3\right)$}

Let in the following be all $3$- and $6$-edges  of color one.
\begin{proposition} \label{prop:U3}
The ring of invariants is a complete intersection of homological dimension two minimally generated by
\begin{center}
\gS{h_a=}
\begin{tikzpicture}[baseline=(A),outer sep=0pt,inner sep=0pt]
\node (A) at (0,0) {};
\sixeu{0}{0}{}
\onee{0}{a}
\end{tikzpicture}
\kern-0.5em
\gS{,a=1,2,3;}
\
\gS{h_{ab}=}
\begin{tikzpicture}[baseline=(A),outer sep=0pt,inner sep=0pt]
\node (A) at (0,0) {};
\thretr{0}{0}{}
\thretl{0.9}{0}{}
\threAB{0}{0.9}{}
\threBAu{0}{0.9}{}
\onee{0}{a}
\oneeu{0.9}{b}
\end{tikzpicture}
\kern-0.5em
\gS{,ab \in\{11,12,13,22,23,33\};}
\\
\gS{\tilde{h}_{ab}=}
\begin{tikzpicture}[baseline=(A),outer sep=0pt,inner sep=0pt]
\node (A) at (0,0) {};
\thretr{0}{0}{}
\threBA{0}{0.7}{}
\oneel{0}{a}
\oneeu{0}{b}
\sixeu{0.7}{0}{}
\end{tikzpicture}
\kern-0.5em
\gS{,ab \in\{12,13,23\};}
\
\gS{h=}
\begin{tikzpicture}[baseline=(A),outer sep=0pt,inner sep=0pt]
\node (A) at (0,0) {};
\thred{0.7}{-0.7}
\threABC{0}{0}{1.4}{0}{0.7}{-0.7}{}
\threABARB{1.4}{0}{0.7}{-0.7}{}{0}{-0.2}
\threABARB{0.7}{-0.7}{0}{0}{}{-0.15}{-0.05}
\threABARB{0}{0}{1.4}{0}{}{0.2}{0.15}
\threu{0}{0}
\threu{1.4}{0}
\end{tikzpicture}
\kern-0.5em
\gS{,}
\
\gS{h_{123}=}
\begin{tikzpicture}[baseline=(A),outer sep=0pt,inner sep=0pt]
\node (A) at (0,0) {};
\thred{0.7}{-0.7}
\threABARB{1.4}{0}{0.7}{-0.7}{}{0}{-0.2}
\threABARB{0.7}{-0.7}{0}{0}{}{-0.15}{-0.05}
\threABARB{0}{0}{1.4}{0}{}{0.2}{0.15}
\threu{0}{0}
\threu{1.4}{0}
\onee{0}{1}
\onee{1.4}{3}
\draw[line width=0.5pt, draw=black] (0.7,-0.7)-- (0.7,-0.35);
\node[font=\tiny,align=center] at (0.7,-0.23) {$\iE{2}$};
\end{tikzpicture}
\kern-0.5em
\gS{,}
\\
\gS{\tilde{h}=}
\begin{tikzpicture}[baseline=(A),outer sep=0pt,inner sep=0pt]
\node (A) at (0,0) {};
\sixeu{0}{0}{}
\sixeu{1.4}{0}{}
\threu{0.7}{0}
\threABARB{0}{0}{0.7}{0}{}{0.1}{-0.15}
\threABARB{1.4}{0}{0.7}{0}{}{-0.1}{-0.15}
\end{tikzpicture}
\kern-0.5em
\gS{,}
\
\gS{h_{123123}=}
\begin{tikzpicture}[baseline=(A),outer sep=0pt,inner sep=0pt]
\node (A) at (0,0) {};
\node[font=\tiny,align=center] at (-0.2,-0.5) {$\iE{1}$};
\node[font=\tiny,align=center] at (0.2,-0.5) {$\iE{3}$};
\draw[line width=0.5pt, draw=black] (0,0)-- (-0.2,-0.38);
\draw[line width=0.5pt, draw=black] (0,0)-- (0.2,-0.38);
\onee{0}{2}
\node[font=\tiny,align=center] at (1.2,-0.5) {$\iE{1}$};
\node[font=\tiny,align=center] at (1.6,-0.5) {$\iE{3}$};
\draw[line width=0.5pt, draw=black] (1.4,0)-- (1.2,-0.38);
\draw[line width=0.5pt, draw=black] (1.4,0)-- (1.6,-0.38);
\onee{1.4}{2}
\threu{0}{0}{}
\threu{0.7}{0}
\threu{1.4}{0}
\threABARB{0}{0}{0.7}{0}{}{0.1}{-0.15}
\threABARB{1.4}{0}{0.7}{0}{}{-0.1}{-0.15}
\end{tikzpicture}
\kern-0.5em
\gS{,}
\
\gS{\tilde{h}_{123}=}
\begin{tikzpicture}[baseline=(A),outer sep=0pt,inner sep=0pt]
\node (A) at (0,0) {};
\node[font=\tiny,align=center] at (-0.2,-0.5) {$\iE{1}$};
\node[font=\tiny,align=center] at (0.2,-0.5) {$\iE{3}$};
\draw[line width=0.5pt, draw=black] (0,0)-- (-0.2,-0.38);
\draw[line width=0.5pt, draw=black] (0,0)-- (0.2,-0.38);
\onee{0}{2}
\sixeu{1.4}{0}{}
\threu{0}{0}{}
\threu{0.7}{0}
\threABARB{0}{0}{0.7}{0}{}{0.1}{-0.15}
\threABARB{1.4}{0}{0.7}{0}{}{-0.1}{-0.15}
\end{tikzpicture}
\gS{.}
\end{center}
\end{proposition}

\begin{proof}
The Krull dimension of the ring of invariants is $15$. The first $15$ invariants can be extracted from~\cite[Table 7]{shmel}. Also from there, we know that an irreducible graph with five $3$-edges and six $1$-edges - two of each color $1,2,3$ - must exist. Since such a graph can not be equivalent to a sum of disconnected graphs due to the lack of appropriate graphs with fewer vertices, any graph with such edges either evaluates to zero or it is irreducible and the irreducible ones are pairwise reducibly equivalent. The penultimate one from the above list is such an irreducible graph.  Remaining graphs from a maximal set of reducibly independent irreducible ones must now contain at least one $1$-, one $6$- and one $1$-edge of each possible color. The last one from the above list is such a graph.

Now we follow the outline of Shmelkin~\cite[pp. 221, 227]{shmel}, where he almost succeeded.
Let $z=e_1\wedge e_2 \wedge e_5+e_3\wedge e_4 \wedge e_6 + e_1 \wedge e_3 \wedge e_7 + e_2 \wedge e_4 \wedge e_7$. Then $\SL_7 z$ is dense in $V(h)$ and the isotropy group of $z$ is $G_2$. Denote by $g_i$ the restrictions of the other generators of $\CC[W]^{\SL_5}$ to $3V+z+\BigWedge^6$.
Consider the obvious $\ZZ^5$ and $\ZZ^4$-gradings of $\CC[W]$ and  $\CC\left[4\CC^7\right]=\CC\left[3V+\BigWedge^6\right]$ respectively.
Then the proof of~\cite[Prop. 4.5]{shmel2} says that the Hilbert series of the  algebras $\CC\left[4\CC^7\right]^{G_2}$ and $\CC\left[g_i\right]$ are identical. Now due to~\cite{schwarz}, the multivariate Hilbert series of the first algebra is
$$
\frac{(1-t_1^2t_2^2t_3^2t_4^2)}{\prod_{1\leq i\leq j \leq 4}(1-t_it_j)\prod_{1\leq k \leq 4}\left(1-t_1t_2t_3t_4t_k^{-1}\right)(1-t_1t_2t_3t_4)}.
$$
Consider all invariants from the proposition but $h$. Let $e_i=e_1\wedge...\wedge e_{i-1}\wedge e_{i+1}...\wedge e_{7}$.
We can restrict to $2V+\CC e_7+z+(e_1^*+\CC e_3^*+\CC e_7^*)$  by~\cite[Ex. on p. 221]{shmel}. If we further restrict to  $2V+e_7+z+(e_1^*+ e_3^*+\CC e_7^*)$, then $\tilde{h}$ and $h_{33}$ become constants. The restrictions - denoted by Fraktur letters - of the remaining invariants satisfy the relations
$$
\mathfrak{h}_{13}^2\mathfrak{h}_{22}-2\mathfrak{h}_{12}\mathfrak{h}_{13}\mathfrak{h}_{23}+\mathfrak{h}_{11}\mathfrak{h}_{23}^2+2\mathfrak{h}_{12}^2-2\mathfrak{h}_{11}\mathfrak{h}_{22}-8\mathfrak{h}_{123}^2,
$$
$$
\mathfrak{h}_3^2\mathfrak{h}_{12}^2-2\mathfrak{h}_{2}\mathfrak{h}_3\mathfrak{h}_{12}\mathfrak{h}_{13}+\mathfrak{h}_{2}^2\mathfrak{h}_{13}^2-\mathfrak{h}_3^2\mathfrak{h}_{11}\mathfrak{h}_{22}+2\mathfrak{h}_{1}\mathfrak{h}_3\mathfrak{h}_{13}\mathfrak{h}_{22}+2\mathfrak{h}_{2}\mathfrak{h}_3\mathfrak{h}_{11}\mathfrak{h}_{23}-2\mathfrak{h}_{1}\mathfrak{h}_3\mathfrak{h}_{12}\mathfrak{h}_{23}
$$
$$
-2\mathfrak{h}_{1}\mathfrak{h}_{2}\mathfrak{h}_{13}\mathfrak{h}_{23}+\mathfrak{h}_{1}^2\mathfrak{h}_{23}^2-2\mathfrak{h}_{2}^2\mathfrak{h}_{11}+4\mathfrak{h}_{1}\mathfrak{h}_{2}\mathfrak{h}_{12}-2\mathfrak{h}_{1}^2\mathfrak{h}_{22}-4\mathfrak{h}_{23}\tilde{\mathfrak{h}}_{12}\tilde{\mathfrak{h}}_{13}+2\mathfrak{h}_{22}\tilde{\mathfrak{h}}_{13}^2
$$
$$
+4\mathfrak{h}_{13}\tilde{\mathfrak{h}}_{12}\tilde{\mathfrak{h}}_{23}-4\mathfrak{h}_{12}\tilde{\mathfrak{h}}_{13}\tilde{\mathfrak{h}}_{23}+2\mathfrak{h}_{11}\tilde{\mathfrak{h}}_{23}^2-8\mathfrak{h}_3\tilde{\mathfrak{h}}_{12}\mathfrak{h}_{123}+8\mathfrak{h}_{2}\tilde{\mathfrak{h}}_{13}\mathfrak{h}_{123}-8\mathfrak{h}_{1}\tilde{\mathfrak{h}}_{23}\mathfrak{h}_{123}+4\tilde{\mathfrak{h}}_{12}^2
$$
$$
-16\tilde{\mathfrak{h}}_{123}^2+32\mathfrak{h}_{123123}.
$$
The first one corresponds to the syzygy of $(\SL_7,3V+\BigWedge^3V)$ from~\cite[Table 7]{shmel}. The second one is new. By suitably multiplying with $\tilde{h}$ and $h_{33}$ and adding $hh_{123123}$, we get the syzygies
$$
h_{13}^2h_{22}-2h_{12}h_{13}h_{23}+h_{11}h_{23}^2+h_{33}h_{12}^2-h_{11}h_{22}h_{33}-8h_{123}^2+hh_{123123},
$$
$$
h_3^2h_{12}^2-2h_{2}h_3h_{12}h_{13}+h_{2}^2h_{13}^2-h_3^2h_{11}h_{22}+2h_{1}h_3h_{13}h_{22}+2h_{2}h_3h_{11}h_{23}-2h_{1}h_3h_{12}h_{23}
$$
$$
-2h_{1}h_{2}h_{13}h_{23}+h_{1}^2h_{23}^2-2h_{2}^2h_{11}h_{33}+4h_{1}h_{2}h_{12}h_{33}-2h_{1}^2h_{22}h_{33}-4h_{23}\tilde{h}_{12}\tilde{h}_{13}+2h_{22}\tilde{h}_{13}^2
$$
$$
+4h_{13}\tilde{h}_{12}\tilde{h}_{23}-4h_{12}\tilde{h}_{13}\tilde{h}_{23}+2h_{11}\tilde{h}_{23}^2-8h_3\tilde{h}_{12}h_{123}+8h_{2}\tilde{h}_{13}h_{123}-8h_{1}\tilde{h}_{23}h_{123}+4\tilde{h}_{12}^2h_{33}
$$
$$
-16\tilde{h}_{123}^2+32h_{123123}\tilde{h}.
$$
The Hilbert series of the algebra with these two syzygies is
$$
\frac{(1-t_1^2t_2^2t_3^2t_4^2)(1-t_1^2t_2^2t_3^2)}{\prod_{1\leq i\leq j \leq 4}(1-t_it_j)\prod_{1\leq k \leq 4}\left(1-t_1t_2t_3t_4t_k^{-1}\right)(1-t_1t_2t_3t_4)(1-t_1^2t_2^2t_3^2)}
$$
and coincides with the one of $\CC\left[4\CC^7\right]^{G_2}$, which proves the assertion.
\end{proof}

\begin{remark}
An alternative way to prove Proposition~\ref{prop:U3} would be to decompose
$$
U_3= \left(\BigWedge^3 V \right) + \left(3 V  + \BigWedge^6 V \right) ,
$$
where for the first $\SL_7$-module, we know a set of generators of the algebra of covariants due to~\cite{RS} and for the second one, such set is easy to determine with the same techniques we apply in Section~\ref{sec:serial}. Following up the approach of Section~\ref{sec:serial}, from these covariant algebras one can explicitly compute a minimal set of generators coinciding with the one from the proposition, find relations between them of suitable multidegree and show that these cut out a complete intersection of the right dimension. 
\end{remark}

\section{The serial cases} \label{sec:serial}

In this section, let always $n=2p+1$. We consider the three $\SL_n$-representations
$
V_1:=\BigWedge^2 V  + 2\BigWedge^{2p} V$
,
$
V_2:= V + \BigWedge^2 V  + \BigWedge^{2p} V$
,
$
V_3:=S^2 V$.
The serial cases that were left open in~\cite{shmel} are
$$
W_1:=2V_1, \quad W_2:=V_1+V_2, \quad W_3:=V_1+V_3.
$$

Our approach is similar in all  cases and involves three steps, which we describe in the following.
For the first step, let us denote the maximal unipotent subgroup of $\SL_n$ of lower triangular matrices by $U$, the \emph{opposite} maximal unipotent subgroup of upper triangular matrices by $U^{o}$ and the normalizing maximal torus by $T$. By Theorem 0.2 of~\cite{panyushev}, the algebra $\CC[X+Y]^{\SL_n}$ is a deformation of $(\CC[X]^U\otimes\CC[Y]^{U^{o}})^T$ for affine varieties $X,Y$ and both algebras share the same Hilbert series with respect to a common $\SL_n$-stable grading. 
We explicitly compute  $A_i:=(\CC[V_1]^U\otimes\CC[V_i]^{U^{o}})^T$ and its Hilbert series for $i=1,2,3$. 
In order to do this, in Lemmata~\ref{le:W1},~\ref{le:W2},~\ref{le:W3} we apply our graph method to algebras of covariants.

Now the explicit form of the $A_i$ not only provides us degrees of potential syzygies for $\CC[W_i]^{\SL_n}$, but also \emph{very important parts} of them, because the syzygies of $A_i$ turn out to be - as one would expect - contractions of syzygies holding in $\CC[W_i]^{\SL_n}$. Not all of them deform to \emph{generators} of the ideal of syzygies of $\CC[W_i]^{\SL_n}$ - those that do not are responsible for $A_i$ not being a complete intersection. But those that do represent \emph{important parts}, because they have no variables in common and suggest a way to prove the ci. property.

In step two, we find explicit forms of syzygies for $\CC[W_i]^{\SL_n}$, which is made possible by our graph theoretic method.

Finally in step three, we prove that the syzygies we found in step two generate the ideal of syzygies.
In all three cases, we find a suitable monomial matrix order such that these syzygies are a Gr\"obner basis with respect to this order and derive the ci. property. The respective leading monomials also occur in the contracted versions of the syzygies in $A_i$. A posteriori, we see that $A_i$ is just not the \textit{best} contraction of $\CC[W_i]^{\SL_n}$ for proving the ci. property: we can deform $A_i$ to some algebra $B_i$ that on the one hand lacks all the superfluous equations and generators of $A_i$ and on the other hand keeps the important parts of the syzygies of $\CC[W_i]^{\SL_n}$.

\subsection{Rings of Covariants}\label{subs:cov}

\begin{lemma}\label{le:covL2+2Ln-1}
The algebra of covariants $\Cov\left(\SL_{2p+1}, V_1\right)$ for $p\geq 2$ is minimally generated by the $4p+3$ covariants

\begin{center}
\gS{c_1=}
\begin{tikzpicture}[baseline=(A),outer sep=0pt,inner sep=0pt]
\node (A) at (0,0) {};
\twoe{0}{1};
\starugd{0}{0}{0.6};
\end{tikzpicture}
\kern-0.5em
\gS{,}
\
\gS{c_2=}
\begin{tikzpicture}[baseline=(A),outer sep=0pt,inner sep=0pt]
\node (A) at (0,0) {};
\twoes{0}{1}{0.5};
\begin{scope}[rotate around={120:(0,0)}]
\twoes{0}{1}{0.5};
\end{scope}
\begin{scope}[rotate around={60:(0,0)}]
\starugd{0}{0}{0.5};
\end{scope}

\end{tikzpicture}
\kern-0.5em
\gS{,\ldots,}
\
\gS{c_{p}=}
\begin{tikzpicture}[baseline=(A),outer sep=0pt,inner sep=0pt]
\node (A) at (0,0) {};

\draw[line width=0.5pt, draw=lightgray] (0,0)-- (0.38,0);
\begin{scope}[rotate around={-35:(0,0)}]
\twoes{0}{1}{0.4};
\end{scope}

\begin{scope}[rotate around={35:(0,0)}]
\twoes{0}{1}{0.4};

\draw[line width=0.7pt, draw=black,dotted]  (-0.2,0.15) .. controls (-0.3,0.05) and (-0.3,-0.05) .. (-0.2,-0.15) ;
\end{scope}
\begin{scope}[rotate around={215:(0,0)}]
\twoes{0}{1}{0.4};
\end{scope}
\end{tikzpicture}
\gS{,}
\\
\gS{c^{(a)}_{1}=}
\begin{tikzpicture}[baseline=(A),outer sep=0pt,inner sep=0pt]
\node (A) at (0,0) {};

\draw[line width=0.5pt, draw=lightgray] (0,0)-- (0,-0.38);
\staru{0}{0}{x}

\end{tikzpicture}
\kern-0.5em
\gS{,}
\
\gS{c^{(a)}_{2}=}
\begin{tikzpicture}[baseline=(A),outer sep=0pt,inner sep=0pt]
\node (A) at (0,0) {};
\staru{0}{0}{x};
\twoeABS{0}{0.7}{1};
\starug{0.7}{0};
\end{tikzpicture}
\kern-0.5em
\gS{,}
\
\gS{c^{(a)}_{3}=}
\begin{tikzpicture}[baseline=(A),outer sep=0pt,inner sep=0pt]
\node (A) at (0,0) {};
\staru{0}{0}{x};
\twoeABS{0}{0.7}{1};
\begin{scope}[shift={(0.7,0)}]
\twoe{0}{1};
\starugd{0}{0}{0.6};
\end{scope}
\end{tikzpicture}
\kern-0.1em
\gS{,$\ldots$,}
\
\gS{c^{(a)}_{p+1}=}
\begin{tikzpicture}[baseline=(A),outer sep=0pt,inner sep=0pt]
\node (A) at (0,0) {};
\staru{0}{0}{x};
\twoeABS{0}{0.7}{1};
\begin{scope}[shift={(0.7,0)}]
\begin{scope}[rotate around={220:(0,0)}]
\twoes{0}{1}{0.4};
\end{scope}

\begin{scope}[rotate around={145:(0,0)}]
\thregs{0}{0.4};
\end{scope}

\begin{scope}[rotate around={30:(0,0)}]
\twoes{0}{1}{0.4};
\end{scope}
\begin{scope}[rotate around={220:(0,0)}]
\draw[line width=0.7pt, draw=black,dotted]  (-0.2,0.15) .. controls (-0.3,0.05) and (-0.3,-0.05) .. (-0.2,-0.15) ;
\end{scope}
\end{scope}
\end{tikzpicture}
\gS{, a=1,2;}
\\
\gS{c^*_1=}
\begin{tikzpicture}[baseline=(A),outer sep=0pt,inner sep=0pt]
\node (A) at (0,0) {};
\draw[line width=0.5pt, draw=lightgray] (0,0)-- (0.7,0);

\staru{0}{0}{1};
\staru{0.7}{0}{2};
\end{tikzpicture}
\kern-0.5em
\gS{,}
\
\gS{c^*_2=}
\begin{tikzpicture}[baseline=(A),outer sep=0pt,inner sep=0pt]
\node (A) at (0,0) {};
\staru{0}{0}{1};
\twoeABS{0}{0.7}{1};
\staru{0.7}{0}{2};
\end{tikzpicture}
\kern-0.5em
\gS{,}
\
\gS{c^*_{3}=}
\begin{tikzpicture}[baseline=(A),outer sep=0pt,inner sep=0pt]
\node (A) at (0,0) {};

\starABg{0}{1.4};
\staru{0}{0}{1};

\staru{0.7}{0}{2};

\twoes{1.4}{1}{0.5};
\begin{scope}[rotate around={240:(1.4,0)}]
\twoes{1.4}{1}{0.5};
\end{scope}

\end{tikzpicture}
\kern-0.5em
\gS{,$\ldots$,}
\
\gS{c^*_{p+1}=}
\begin{tikzpicture}[baseline=(A),outer sep=0pt,inner sep=0pt]
\node (A) at (0,0) {};

\threABg{0}{1.4};
\staru{0}{0}{1};

\staru{0.7}{0}{2};
\begin{scope}[shift={(1.4,0)}]
\begin{scope}[rotate around={220:(0,0)}]
\twoes{0}{1}{0.4};
\end{scope}

\begin{scope}[rotate around={30:(0,0)}]
\twoes{0}{1}{0.4};
\end{scope}
\begin{scope}[rotate around={220:(0,0)}]
\draw[line width=0.7pt, draw=black,dotted]  (-0.2,0.15) .. controls (-0.3,0.05) and (-0.3,-0.05) .. (-0.2,-0.15) ;
\end{scope}
\end{scope}

\end{tikzpicture}
\kern-0.5em
\gS{.}
\end{center}
Moreover, the ideal of relations is generated by 
$$
c_p \,  c^*_1 - p \, c^{(1)}_{1} \, c^{(2)}_{p+1} + p \, c^{(2)}_{1} \, c^{(1)}_{p+1}.
$$
\end{lemma}

\begin{proof}
Due to~\cite[Thm. 3]{RS}, covariant graphs are similar to invariant graphs but can in addition contain looping dummy $k$-edges, behaving as if they correspond to additional copies of $\BigWedge^k V$. In fact, the algebra of covariants is isomorphic to the algebra of invariants $\CC[W]^{U(G)}$ for a maximal unipotent subgroup $U(G)$ of $G$, see~\cite{shmel3}. The isomorphism is given  - if we choose the upper triangular matrices for $U(G)$ - by evaluation at $e_1 \wedge \ldots \wedge e_k$ for each dummy $k$-edge. 
The multidegrees of the covariants can be deduced from~\cite[Table 1]{brionRed}.
The explicit forms of the graphs for these covariants are then obvious.

Since the Krull dimension of $\CC[W]^{U(G)}$ is $4p+2$, see~\cite{brionRed}, we have one syzygy. 
We obtain this syzygy by considering the graph 
\begin{center}
\begin{tikzpicture}[baseline=(A),outer sep=0pt,inner sep=0pt]
\node (A) at (0,0) {};
\staru{0}{0}{1};
\staru{1.4}{0}{2};
\twoeABS{0}{0.7}{1};

\twoeABS{0.7}{1.4}{1};
\begin{scope}[shift={(0.7,0)}]
\begin{scope}[rotate around={235:(0,0)}]
\twoes{0}{1}{0.3};
\end{scope}

\begin{scope}[rotate around={-30:(0,0)}]
\thregs{0}{0.3};
\end{scope}

\begin{scope}[rotate around={30:(0,0)}]
\draw[line width=0.5pt, draw=lightgray] (0,0)-- (0,0.38);
\end{scope}

\begin{scope}[rotate around={125:(0,0)}]
\twoes{0}{1}{0.3};
\end{scope}

\draw[line width=0.7pt, draw=black,dotted]  (0.1,-0.3) -- (-0.1,-0.3) ;

\end{scope}
\end{tikzpicture}
\gS{.}
\end{center}
Applying the Pl\"ucker relation to the left non-looping $2$-edge gives:
\begin{center}
\gS{(p-1)}
\begin{tikzpicture}[baseline=(A),outer sep=0pt,inner sep=0pt]
\node (A) at (0,0) {};
\staru{0}{0}{1};
\staru{1.4}{0}{2};
\twoeABS{0}{0.7}{1};

\twoeABS{0.7}{1.4}{1};
\begin{scope}[shift={(0.7,0)}]
\begin{scope}[rotate around={235:(0,0)}]
\twoes{0}{1}{0.3};
\end{scope}

\begin{scope}[rotate around={-30:(0,0)}]
\thregs{0}{0.3};
\end{scope}

\begin{scope}[rotate around={30:(0,0)}]
\draw[line width=0.5pt, draw=lightgray] (0,0)-- (0,0.38);
\end{scope}

\begin{scope}[rotate around={125:(0,0)}]
\twoes{0}{1}{0.3};
\end{scope}

\draw[line width=0.7pt, draw=black,dotted]  (0.1,-0.3) -- (-0.1,-0.3) ;

\end{scope}
\end{tikzpicture}
\gS{=}
\begin{tikzpicture}[baseline=(A),outer sep=0pt,inner sep=0pt]
\node (A) at (0,0) {};
\staru{0}{0}{1};
\twoeABS{0}{0.7}{1};
\staru{0.7}{0}{2};
\end{tikzpicture}
\kern-0.3em
\begin{tikzpicture}[baseline=(A),outer sep=0pt,inner sep=0pt]
\node (A) at (0,0) {};

\draw[line width=0.5pt, draw=lightgray] (0,0)-- (0.38,0);
\begin{scope}[rotate around={-35:(0,0)}]
\thregs{0}{0.4};
\end{scope}

\begin{scope}[rotate around={35:(0,0)}]
\twoes{0}{1}{0.4};

\draw[line width=0.7pt, draw=black,dotted]  (-0.2,0.15) .. controls (-0.3,0.05) and (-0.3,-0.05) .. (-0.2,-0.15) ;
\end{scope}
\begin{scope}[rotate around={215:(0,0)}]
\twoes{0}{1}{0.4};
\end{scope}
\end{tikzpicture}
\gS{+}
\begin{tikzpicture}[baseline=(A),outer sep=0pt,inner sep=0pt]
\node (A) at (0,0) {};
\draw[line width=0.5pt, draw=lightgray] (0,0)-- (0,-0.38);
\staru{0}{0}{1}
\end{tikzpicture}
\kern-0.1em
\begin{tikzpicture}[baseline=(A),outer sep=0pt,inner sep=0pt]
\node (A) at (0,0) {};
\staru{0}{0}{2};
\twoeABS{0}{0.7}{1};
\begin{scope}[shift={(0.7,0)}]
\begin{scope}[rotate around={220:(0,0)}]
\twoes{0}{1}{0.4};
\end{scope}

\begin{scope}[rotate around={145:(0,0)}]
\thregs{0}{0.4};
\end{scope}

\begin{scope}[rotate around={30:(0,0)}]
\twoes{0}{1}{0.4};
\end{scope}
\begin{scope}[rotate around={220:(0,0)}]
\draw[line width=0.7pt, draw=black,dotted]  (-0.2,0.15) .. controls (-0.3,0.05) and (-0.3,-0.05) .. (-0.2,-0.15) ;
\end{scope}
\end{scope}
\end{tikzpicture}
\gS{-}
\begin{tikzpicture}[baseline=(A),outer sep=0pt,inner sep=0pt]
\node (A) at (0,0) {};
\draw[line width=0.5pt, draw=lightgray] (0,0)-- (0.7,0);
\staru{0}{0}{1};
\staru{1.4}{0}{2};

\twoeABS{0.7}{1.4}{1};
\begin{scope}[shift={(0.7,0)}]
\begin{scope}[rotate around={235:(0,0)}]
\twoes{0}{1}{0.3};
\end{scope}

\begin{scope}[rotate around={30:(0,0)}]
\twoes{0}{1}{0.3};
\end{scope}

\begin{scope}[rotate around={-30:(0,0)}]
\draw[line width=0.5pt, draw=lightgray] (0,0)-- (0,0.38);
\end{scope}

\begin{scope}[rotate around={125:(0,0)}]
\twoes{0}{1}{0.3};
\end{scope}

\draw[line width=0.7pt, draw=black,dotted]  (0.1,-0.3) -- (-0.1,-0.3) ;

\end{scope}
\end{tikzpicture}
\kern-0.3em
\gS{.}
\end{center}
Now in the last graph on the right, we apply the Pl\"ucker relation to the non-looping $2$-edge of color one and get:
\begin{center}
\gS{(p-1)}
\begin{tikzpicture}[baseline=(A),outer sep=0pt,inner sep=0pt]
\node (A) at (0,0) {};
\staru{0}{0}{1};
\staru{1.4}{0}{2};
\twoeABS{0}{0.7}{1};

\twoeABS{0.7}{1.4}{1};
\begin{scope}[shift={(0.7,0)}]
\begin{scope}[rotate around={235:(0,0)}]
\twoes{0}{1}{0.3};
\end{scope}

\begin{scope}[rotate around={-30:(0,0)}]
\thregs{0}{0.3};
\end{scope}

\begin{scope}[rotate around={30:(0,0)}]
\draw[line width=0.5pt, draw=lightgray] (0,0)-- (0,0.38);
\end{scope}

\begin{scope}[rotate around={125:(0,0)}]
\twoes{0}{1}{0.3};
\end{scope}

\draw[line width=0.7pt, draw=black,dotted]  (0.1,-0.3) -- (-0.1,-0.3) ;

\end{scope}
\end{tikzpicture}
\gS{=}
\begin{tikzpicture}[baseline=(A),outer sep=0pt,inner sep=0pt]
\node (A) at (0,0) {};
\staru{0}{0}{1};
\twoeABS{0}{0.7}{1};
\staru{0.7}{0}{2};
\end{tikzpicture}
\kern-0.3em
\begin{tikzpicture}[baseline=(A),outer sep=0pt,inner sep=0pt]
\node (A) at (0,0) {};

\draw[line width=0.5pt, draw=lightgray] (0,0)-- (0.38,0);
\begin{scope}[rotate around={-35:(0,0)}]
\thregs{0}{0.4};
\end{scope}

\begin{scope}[rotate around={35:(0,0)}]
\twoes{0}{1}{0.4};

\draw[line width=0.7pt, draw=black,dotted]  (-0.2,0.15) .. controls (-0.3,0.05) and (-0.3,-0.05) .. (-0.2,-0.15) ;
\end{scope}
\begin{scope}[rotate around={215:(0,0)}]
\twoes{0}{1}{0.4};
\end{scope}
\end{tikzpicture}
\kern-0.3em
\gS{+}
\kern-0.3em
\begin{tikzpicture}[baseline=(A),outer sep=0pt,inner sep=0pt]
\node (A) at (0,0) {};
\draw[line width=0.5pt, draw=lightgray] (0,0)-- (0,-0.38);
\staru{0}{0}{1}
\end{tikzpicture}
\kern-0.1em
\begin{tikzpicture}[baseline=(A),outer sep=0pt,inner sep=0pt]
\node (A) at (0,0) {};
\staru{0}{0}{2};
\twoeABS{0}{0.7}{1};
\begin{scope}[shift={(0.7,0)}]
\begin{scope}[rotate around={220:(0,0)}]
\twoes{0}{1}{0.4};
\end{scope}

\begin{scope}[rotate around={145:(0,0)}]
\thregs{0}{0.4};
\end{scope}

\begin{scope}[rotate around={30:(0,0)}]
\twoes{0}{1}{0.4};
\end{scope}
\begin{scope}[rotate around={220:(0,0)}]
\draw[line width=0.7pt, draw=black,dotted]  (-0.2,0.15) .. controls (-0.3,0.05) and (-0.3,-0.05) .. (-0.2,-0.15) ;
\end{scope}
\end{scope}
\end{tikzpicture}
\kern-0.7em
\gS{-}
\kern-0.3em
\begin{tikzpicture}[baseline=(A),outer sep=0pt,inner sep=0pt]
\node (A) at (0,0) {};
\draw[line width=0.5pt, draw=lightgray] (0,0)-- (0,-0.38);
\staru{0}{0}{2}
\end{tikzpicture}
\kern-0.1em
\begin{tikzpicture}[baseline=(A),outer sep=0pt,inner sep=0pt]
\node (A) at (0,0) {};
\staru{0}{0}{1};
\twoeABS{0}{0.7}{1};
\begin{scope}[shift={(0.7,0)}]
\begin{scope}[rotate around={220:(0,0)}]
\twoes{0}{1}{0.4};
\end{scope}

\begin{scope}[rotate around={145:(0,0)}]
\thregs{0}{0.4};
\end{scope}

\begin{scope}[rotate around={30:(0,0)}]
\twoes{0}{1}{0.4};
\end{scope}
\begin{scope}[rotate around={220:(0,0)}]
\draw[line width=0.7pt, draw=black,dotted]  (-0.2,0.15) .. controls (-0.3,0.05) and (-0.3,-0.05) .. (-0.2,-0.15) ;
\end{scope}
\end{scope}
\end{tikzpicture}
\kern-0.5em
\gS{- \frac{1}{p}}
\begin{tikzpicture}[baseline=(A),outer sep=0pt,inner sep=0pt]
\node (A) at (0,0) {};
\draw[line width=0.5pt, draw=lightgray] (0,0)-- (0.7,0);

\staru{0}{0}{1};
\staru{0.7}{0}{2};
\end{tikzpicture}
\kern-0.3em
\begin{tikzpicture}[baseline=(A),outer sep=0pt,inner sep=0pt]
\node (A) at (0,0) {};

\draw[line width=0.5pt, draw=lightgray] (0,0)-- (0.38,0);
\begin{scope}[rotate around={-35:(0,0)}]
\twoes{0}{1}{0.4};
\end{scope}

\begin{scope}[rotate around={35:(0,0)}]
\twoes{0}{1}{0.4};

\draw[line width=0.7pt, draw=black,dotted]  (-0.2,0.15) .. controls (-0.3,0.05) and (-0.3,-0.05) .. (-0.2,-0.15) ;
\end{scope}
\begin{scope}[rotate around={215:(0,0)}]
\twoes{0}{1}{0.4};
\end{scope}
\end{tikzpicture}
\kern-0.7em
\gS{.}
\end{center}
But due to the isomorphism from above, we can evaluate at $e_1$ and $e_1 \wedge e_2$ for the dummy $1$- and $2$-edge respectively and see that the two leftmost graphs disappear. We arrive at the desired relation.

\end{proof}

\begin{remark}
Observe that in Lemma~\ref{le:covL2+2Ln-1}, we did not always choose the graph with a \emph{looping} dummy edge for the invariant, but it is easy to deduce one version from another by applying the Pl\"ucker relation to the dummy edge.
\end{remark}

\begin{lemma}\label{le:covL+L2+Ln-1}
The algebra of covariants $\Cov(\SL_{2p+1}, V_2)$ for $p\geq 2$ is polynomial, generated by the $c_1,\ldots,c_p$ and $c^{(1)}_{1},\ldots,c^{(1)}_{p+1}$ from Lemma~\ref{le:covL2+2Ln-1} and in addition by the covariants
\begin{center}
\gS{c^\circ_1=}
\begin{tikzpicture}[baseline=(A),outer sep=0pt,inner sep=0pt]
\node (A) at (0,0) {};
\begin{scope}[rotate around={90:(0,0)}]
\onee{0}{1};
\end{scope}
\begin{scope}[rotate around={-90:(0,0)}]
\starugd{0}{0}{0.7};
\end{scope}
\end{tikzpicture}
\kern-0.1em
\gS{,}
\
\gS{c^\circ_2=}
\begin{tikzpicture}[baseline=(A),outer sep=0pt,inner sep=0pt]
\node (A) at (0,0) {};
\twoe{0}{1};
\starugd{0}{0}{0.6};
\begin{scope}[rotate around={90:(0,0)}]
\onee{0}{1};
\end{scope}
\end{tikzpicture}
\kern-0.1em
\gS{,\ldots,}
\
\gS{c^\circ_{p+1}=}
\begin{tikzpicture}[baseline=(A),outer sep=0pt,inner sep=0pt]
\node (A) at (0,0) {};

\begin{scope}[rotate around={-35:(0,0)}]
\twoes{0}{1}{0.4};
\end{scope}

\begin{scope}[rotate around={35:(0,0)}]
\twoes{0}{1}{0.4};

\draw[line width=0.7pt, draw=black,dotted]  (-0.2,0.15) .. controls (-0.3,0.05) and (-0.3,-0.05) .. (-0.2,-0.15) ;
\end{scope}
\begin{scope}[rotate around={215:(0,0)}]
\twoes{0}{1}{0.4};
\end{scope}
\begin{scope}[rotate around={90:(0,0)}]
\onee{0}{1};
\end{scope}
\end{tikzpicture}
\gS{,}
\\
\gS{c^\Diamond_{1}=}
\begin{tikzpicture}[baseline=(A),outer sep=0pt,inner sep=0pt]
\node (A) at (0,0) {};

\onee{0}{1}
\staru{0}{0}{1}

\end{tikzpicture}
\kern-0.5em
\gS{,}
\gS{c^\Diamond_{2}=}
\begin{tikzpicture}[baseline=(A),outer sep=0pt,inner sep=0pt]
\node (A) at (0,0) {};
\staru{0}{0}{1};
\twoeABS{0}{0.7}{1};
\starug{0.7}{0};
\onee{0.7}{1}
\end{tikzpicture}
\kern-0.5em
\gS{,}
\
\gS{c^\Diamond_{3}=}
\begin{tikzpicture}[baseline=(A),outer sep=0pt,inner sep=0pt]
\node (A) at (0,0) {};
\staru{0}{0}{1};
\twoeABS{0}{0.7}{1};
\begin{scope}[shift={(0.7,0)}]
\begin{scope}[rotate around={90:(0,0)}]
\starugd{0}{0}{0.4};
\end{scope}
\twoes{0}{1}{0.4};

\end{scope}
\onee{0.7}{1}
\end{tikzpicture}
\kern-0.1em
\gS{,$\ldots$,}
\
\gS{c^\Diamond_{p}=}
\begin{tikzpicture}[baseline=(A),outer sep=0pt,inner sep=0pt]
\node (A) at (0,0) {};
\staru{0}{0}{1};
\twoeABS{0}{0.7}{1};
\begin{scope}[shift={(0.7,0)}]
\begin{scope}[rotate around={230:(0,0)}]
\twoes{0}{1}{0.4};
\end{scope}

\begin{scope}[rotate around={135:(0,0)}]
\draw[line width=0.5pt, draw=lightgray] (0,0)-- (0,0.38);
\thregs{0}{0.4};
\end{scope}

\begin{scope}[rotate around={30:(0,0)}]
\twoes{0}{1}{0.4};
\end{scope}
\begin{scope}[rotate around={225:(0,0)}]
\draw[line width=0.7pt, draw=black,dotted]  (-0.2,0.15) .. controls (-0.3,0.05) and (-0.3,-0.05) .. (-0.2,-0.15) ;
\end{scope}
\end{scope}
\onee{0.7}{1}
\end{tikzpicture}
\kern-0.5em
\gS{.}
\end{center}

\end{lemma}

\begin{proof}
Due to the equivalence with the algebra of invariants of $U(\SL_{2p+1})$, we get multidegrees of covariants from~\cite[Table 1]{brionRed}, which in addition gives polynomiality. It is straightforward to find the only possible covariants of the matching multidegrees.
\end{proof}

\begin{lemma}\label{le:covL^2}
The algebra of covariants $\Cov(\SL_{n}, V_3)$ is polynomial, generated by the $n$ covariants
\begin{center}
\gS{c^\triangle_{1}=}
\begin{tikzpicture}[baseline=(A),outer sep=0pt,inner sep=0pt]
\node (A) at (0,0) {};
\starugd{0}{0}{0.6};
\starugd{0.7}{0}{0.6};
\twoeS{0}{0.7}{1}{0}
\draw[black, fill=black] (0,0) circle [radius=1.5pt];
\draw[black, fill=black] (0.7,0) circle [radius=1.5pt];
\end{tikzpicture}
\kern-0.1em
\gS{,}
\gS{c^\triangle_{2}=}
\begin{tikzpicture}[baseline=(A),outer sep=0pt,inner sep=0pt]
\node (A) at (0,0) {};

\twoeS{0}{0.7}{1}{0.1};
\twoeSd{0}{0.7}{1}{0.1};

\starugd{0}{0}{0.45};
\starugd{0.7}{0}{0.45};
\draw[black, fill=black] (0,0) circle [radius=1.5pt];
\draw[black, fill=black] (0.7,0) circle [radius=1.5pt];
\end{tikzpicture}
\kern-0.1em
\gS{,$\ldots$,}
\
\gS{c^\triangle_{n}=}
\begin{tikzpicture}[baseline=(A),outer sep=0pt,inner sep=0pt]
\node (A) at (0,0) {};

\twoeS{0}{0.7}{1}{0.35};
\twoeSd{0}{0.7}{1}{0.35};
\draw[line width=0.7pt, draw=black,dotted]  (0.35,0.15) --(0.35,-0.15) ;

\draw[black, fill=black] (0,0) circle [radius=1.5pt];
\draw[black, fill=black] (0.7,0) circle [radius=1.5pt];
\end{tikzpicture}
\kern-0.1em
\gS{.}
\end{center}

\end{lemma}
\begin{proof}
Here, we get polynomiality and multidegrees of covariants from~\cite[Thm. 3]{brionIrred} and proceed as in Lemma~\ref{le:covL+L2+Ln-1}.
\end{proof}

\subsection{The case $(\SL_{2p+1},W_1)$}

\begin{proposition}\label{prop:W1}
The ring of invariants $\CC[W_1]^{\SL_{2p+1}}$ is a complete intersection. Its homological dimension is four if $p=2$ and $2p-4$ if $p \geq 3$. It is minimally generated by
\begin{center}
\gS{\mathfrak{i}_{abc}=}
\begin{tikzpicture}[baseline=(A),outer sep=0pt,inner sep=0pt]
\node (A) at (0,0) {};
\staru{0}{0}{a};
\twoeABS{0}{0.7}{c};
\staru{0.7}{0}{b};
\end{tikzpicture}
\kern-0.5em
\gS{,ab \in \{12,13,14,23,24,34\}, c \in \{1,2\};}
\\
\gS{\mathfrak{j}_{abc}=}
\begin{tikzpicture}[baseline=(A),outer sep=0pt,inner sep=0pt]
\clip (-0.9,-1.3) rectangle (0.9,1.1);
\node (A) at (0,0) {};

\twostr{-0.2}{0.2}{1};
\draw[line width=0.7pt, draw=black,dotted]  (-0.45,0.4) -- (-0.25,0.4) ;
\twostr{-1}{0.4}{};
\node[font=\tiny,align=center] at (-0.8,0.6) {$\iiE{1}$};
\draw[decorate, decoration={brace}, line width=0.5pt, draw=black]  (-0.8,0.75) -- (-0.2,0.75) node[font=\tiny,midway, above=3pt]  {$a$};

\twostr{0.2}{0.2}{2};
\draw[line width=0.7pt, draw=black,dotted]  (0.45,0.4) -- (0.25,0.4) ;
\twostr{1}{0.4}{};
\node[font=\tiny,align=center] at (0.8,0.6) {$\iiE{2}$};
\draw[decorate, decoration={brace,mirror}, line width=0.5pt, draw=black]  (0.8,0.75) -- (0.2,0.75) node[font=\tiny,midway, above=3pt]  {$b-1$};

\begin{scope}[rotate around={180:(0,-0.7)}]
\staru{0}{-0.7}{c};
\end{scope}
\twoeARB{0}{-0.7}{0}{0}{2}{right=1pt};
\end{tikzpicture}
\kern-0.1em
\gS{, a,b\geq 1,  c \in \{1,2,3,4\}; }
\\
\gS{\mathfrak{k}_{abc}=}
\begin{tikzpicture}[baseline=(A),outer sep=0pt,inner sep=0pt]
\clip (-0.9,-1.3) rectangle (1,1.1);
\node (A) at (0,0) {};

\twostr{-0.2}{0.2}{1};
\draw[line width=0.7pt, draw=black,dotted]  (-0.45,0.4) -- (-0.25,0.4) ;
\twostr{-1}{0.4}{};
\node[font=\tiny,align=center] at (-0.8,0.6) {$\iiE{1}$};
\draw[decorate, decoration={brace}, line width=0.5pt, draw=black]  (-0.8,0.75) -- (-0.2,0.75) node[font=\tiny,midway, above=3pt]  {$a$};

\twostr{0.2}{0.2}{2};
\draw[line width=0.7pt, draw=black,dotted]  (0.45,0.4) -- (0.25,0.4) ;
\twostr{1}{0.4}{};
\node[font=\tiny,align=center] at (0.8,0.6) {$\iiE{2}$};
\draw[decorate, decoration={brace,mirror}, line width=0.5pt, draw=black]  (0.8,0.75) -- (0.2,0.75) node[font=\tiny,midway, above=3pt]  {$b-3$};

\begin{scope}[rotate around={180:(0,-0.7)}]
\staru{0}{-0.7}{e};
\end{scope}
\twoeARB{0}{-0.7}{0}{0}{2}{right=1pt};

\begin{scope}[rotate around={180:(-0.7,0)}]
\staru{-0.7}{0}{d};
\end{scope}
\twoeARB{-0.7}{0}{0}{0}{2}{below=2pt};
\begin{scope}[rotate around={180:(0.7,0)}]
\staru{0.7}{0}{f};
\end{scope}
\twoeARB{0.7}{0}{0}{0}{2}{below=2pt};
\end{tikzpicture}
\gS{, a,b\geq 3,  def \in \{234,134,124,123\}, c=\{1,2,3,4\}\setminus\{d,e,f\}. }
\end{center}

\end{proposition}

\begin{lemma}\label{le:W1}
 Denote by $x$ the images of the covariants $c$ from Lemma~\ref{le:covL2+2Ln-1} in $\CC[V_1]^U$ and by $y$ those in $\CC[V_1]^{U^o}$.  The algebra  $A_1=(\CC[V_1]^U\otimes\CC[V_1]^{U^o})^T$ is minimally generated by $x^*_2, y^*_2$ and the entries of the matrices
 \setlength{\arraycolsep}{0.5mm}
 $$
 f^{(1)}:=\begin{bmatrix}
 x^{(1)}_2 \\
 x^{(2)}_2
 \end{bmatrix}
 \cdot
 \begin{bmatrix}
 y^{(1)}_1 & y^{(2)}_1 & y_p 
 \end{bmatrix}
 ,
 \quad
 f^{(2)}:=\begin{bmatrix}
 x_1 \\
 x_3^*
 \end{bmatrix}
 \cdot
 \begin{bmatrix}
 y^{(1)}_{p+1} & y^{(2)}_{p+1} & y_1^*
 \end{bmatrix}
 ,
 $$
 $$
 g^{(1)}:=\begin{bmatrix}
 y^{(1)}_2 \\
 y^{(2)}_2
 \end{bmatrix}
 \cdot
 \begin{bmatrix}
 x^{(1)}_1 & x^{(2)}_1 & x_p 
 \end{bmatrix}
 ,
 \quad
 g^{(2)}:=\begin{bmatrix}
 y_1 \\
 y_3^*
 \end{bmatrix}
 \cdot
 \begin{bmatrix}
 x^{(1)}_{p+1} & x^{(2)}_{p+1} & x_1^*
 \end{bmatrix}
 $$
 $$
h^{(k)}:=\begin{bmatrix}
 x_{2+k} \\
 x^*_{4+k}
 \end{bmatrix}
 \cdot
 \begin{bmatrix}
 y^{(1)}_{p-k} & y^{(2)}_{p-k} 
 \end{bmatrix}
 ,
 \quad
 i^{(k)}:=\begin{bmatrix}
 y_{2+k} \\
 y^*_{4+k}
 \end{bmatrix}
 \cdot
 \begin{bmatrix}
 x^{(1)}_{p-k} & x^{(2)}_{p-k} 
 \end{bmatrix}
 ,
 \quad
  k=0,\ldots,p-3.
 $$
The ideal of syzygies is generated by all $2\times 2$-minors  of the matrices $f^{(j)}$, $g^{(j)}$, $h^{(k)}$, $i^{(k)}$ and
$$
f^{(1)}_{a3}f^{(2)}_{b3}-p f^{(1)}_{a1}f^{(2)}_{b2}+p f^{(1)}_{a2}f^{(2)}_{b1}, 
\,
g^{(1)}_{a3}g^{(2)}_{b3}-p g^{(1)}_{a1}g^{(2)}_{b2}+p g^{(1)}_{a2}g^{(2)}_{b1},
\quad a,b \in \{1,2\}.
$$ 

 \end{lemma}

\begin{proof}
Under the isomorphism between $\Cov(\SL_n,V_1)$ and $\CC[V_1]^{U}$, the images of the covariants from Lemma~\ref{le:covL2+2Ln-1} have weights according to~\cite[Table 1]{brionRed}. From this, we directly deduce generators of the ring of invariants and the ideal of syzygies, similar as in~\cite[Proof of Thm. 1.8]{ltpticr}.

\end{proof}

 \begin{proof}[Proof of Prop.~\ref{prop:W1}]
 The list of generators can be extracted from~\cite[Le. 2.5]{shmel}, it is minimal due to~\cite[Proof of Thm. 0.2]{shmel}.
 We come to the syzygies.
 Let us first have a look at the Hilbert series of $A_1$ with respect to the $\ZZ^6$-grading of $W_1$, where variables $t_i$ correspond to copies of $\BigWedge^2V$ and $s_i$ to $\BigWedge^{n-1}V$ respectively. Lemma~\ref{le:W1} provides us with generators and syzygies. We have 
\begin{align*}
 A_1 = 
\CC[x^*_2, y^*_2] & \otimes  \CC[x_1,x^*_3,x^{(1)}_2,x^{(2)}_2,y_p,y^*_1,y^{(1)}_1,y^{(2)}_1,y^{(1)}_{p+1},y^{(2)}_{p+1}]^T  \\
 & \otimes  \CC[y_1,y^*_3,y^{(1)}_2,y^{(2)}_2,x_p,x^*_1,x^{(1)}_1,x^{(2)}_1,x^{(1)}_{p+1},x^{(2)}_{p+1}]^T \\
 & \otimes  \bigotimes_{k=0}^{p-3}\CC[x_{2+k},x^*_{4+k},y^{(1)}_{p-k},y^{(2)}_{p-k}]^T \\  
  & \otimes \bigotimes_{k=0}^{p-3}\CC[y_{2+k},y^*_{4+k},x^{(1)}_{p-k},x^{(2)}_{p-k}]^T.
    \end{align*}
The first three algebras have the Hilbert series
$${\scriptstyle
\frac{1}{(1-t_1s_1s_2)(1-t_2s_3s_4)},
}
$$
$$
{\scriptstyle
\frac{(1-t_1^2t_2^ps_1s_3s_4)(1-t_1^2t_2^ps_2s_3s_4)}{(1-t_1s_1s_3)(1-t_1s_1s_4)(1-t_1s_2s_3)(1-t_1s_2s_4)(1-t_1t_2^ps_1)(1-t_1t_2^ps_2)(1-t_1t_2^ps_3)(1-t_1t_2^ps_4)(1-t_1s_3s_4)},
}
$$   
 $$
{\scriptstyle
\frac{(1-t_1^pt_2^2s_1s_2s_3)(1-t_1^pt_2^2s_1s_2s_4)}{(1-t_2s_1s_3)(1-t_2s_2s_3)(1-t_2s_1s_4)(1-t_2s_2s_4)(1-t_1^pt_2s_3)(1-t_1^pt_2s_4)(1-t_1^pt_2s_1)(1-t_1^pt_2s_2)(1-t_2s_1s_2)},
}
$$ 
while the $k$-th factors of the fourth and fifth algebra have the respective Hilbert series
$$
{\scriptstyle
\frac{(1-t_1^{5+2k}t_2^{2p-2k-2}s_1s_2s_3s_4)}{(1-t_1^{2+k}t_2^{p-k-1}s_3)(1-t_1^{2+k}t_2^{p-k-1}s_4)(1-t_1^{3+k}t_2^{p-k-1}s_1s_2s_3)(1-t_1^{3+k}t_2^{p-k-1}s_1s_2s_4)},
}
$$
$$
{\scriptstyle
\frac{(1-t_2^{5+2k}t_1^{2p-2k-2}s_1s_2s_3s_4)}{(1-t_2^{2+k}t_1^{p-k-1}s_1)(1-t_2^{2+k}t_1^{p-k-1}s_2)(1-t_2^{3+k}t_1^{p-k-1}s_1s_3s_4)(1-t_2^{3+k}t_1^{p-k-1}s_2s_3s_4)}.
}
$$

Now we distinguish between the cases $p=2,3,\geq 4$ as they behave differently.
For $p=2$, the Hilbert series of $\CC[W_1]^{\SL_5}$ is the product of those of the first three algebras from above. Each factor in the denominator corresponds to one generator. Now  consider the graphs
\begin{center}
\begin{tikzpicture}[baseline=(A),outer sep=0pt,inner sep=0pt]

\node (A) at (0,0) {};

\twoeARB{0}{0.4}{0}{0}{2}{right=1pt};
\foureu{-0.7}{0}{a}
\foureu{0}{0.4}{b}
\foureu{0.7}{0}{c}
\twoeABS{0}{-0.7}{1}
\twoeABS{0}{0.7}{2}
\twoed{0}{1}
\end{tikzpicture}
\gS{,abc \in \{234,134,124,123\}.}
\end{center}
Applying the Pl\"ucker relation to the two edges of color $1$, we get 
\begin{center}
\gS{2}
\begin{tikzpicture}[baseline=(A),outer sep=0pt,inner sep=0pt]

\node (A) at (0,0) {};

\twoeARB{0}{0.4}{0}{0}{2}{right=1pt};
\foureu{-0.7}{0}{a}
\foureu{0}{0.4}{b}
\foureu{0.7}{0}{c}
\twoeABS{0}{-0.7}{1}
\twoeABS{0}{0.7}{2}
\twoed{0}{1}
\end{tikzpicture}
\gS{=}
\gS{-}
\begin{tikzpicture}[baseline=(A),outer sep=0pt,inner sep=0pt]
\node (A) at (0,0) {};
\foureu{0}{0}{a}
\twoeABS{0}{0.7}{2}
\foureu{0.7}{0}{b}
\end{tikzpicture}
\begin{tikzpicture}[baseline=(A),outer sep=0pt,inner sep=0pt]
\node (A) at (0,0) {};
\foureu{0.7}{0}{c}
\twoeABS{0}{0.7}{2}
\twoe{0}{1}
\twoed{0}{1}
\end{tikzpicture}
\gS{+}
\begin{tikzpicture}[baseline=(A),outer sep=0pt,inner sep=0pt]
\node (A) at (0,0) {};
\foureu{0}{0}{a}
\twoeABS{0}{0.7}{2}
\foureu{0.7}{0}{c}
\end{tikzpicture}
\begin{tikzpicture}[baseline=(A),outer sep=0pt,inner sep=0pt]
\node (A) at (0,0) {};
\foureu{0.7}{0}{b}
\twoeABS{0}{0.7}{2}
\twoe{0}{1}
\twoed{0}{1}
\end{tikzpicture}
\gS{.}
\end{center}
On the other hand, if we apply the Pl\"ucker relation to the right edge of color 2, we get 
\begin{center}
\begin{tikzpicture}[baseline=(A),outer sep=0pt,inner sep=0pt]

\node (A) at (0,0) {};

\twoeARB{0}{0.4}{0}{0}{2}{right=1pt};
\foureu{-0.7}{0}{a}
\foureu{0}{0.4}{b}
\foureu{0.7}{0}{c}
\twoeABS{0}{-0.7}{1}
\twoeABS{0}{0.7}{2}
\twoed{0}{1}
\end{tikzpicture}
\gS{=}
\gS{-}
\begin{tikzpicture}[baseline=(A),outer sep=0pt,inner sep=0pt]

\node (A) at (0,0) {};

\twoeARB{0}{0.4}{0}{0}{2}{right=1pt};
\foureu{-0.7}{0}{a}
\foureu{0}{0.4}{b}
\foureu{0.7}{0}{c}
\twoeABS{0}{-0.7}{1}
\twoeABS{0}{0.7}{1}
\twoed{0}{2}
\end{tikzpicture}
\gS{+}
\begin{tikzpicture}[baseline=(A),outer sep=0pt,inner sep=0pt]
\node (A) at (0,0) {};
\foureu{0}{0}{a}
\twoeABS{0}{0.7}{1}
\foureu{0.7}{0}{c}
\end{tikzpicture}
\begin{tikzpicture}[baseline=(A),outer sep=0pt,inner sep=0pt]
\node (A) at (0,0) {};
\foureu{0.7}{0}{b}
\twoeABS{0}{0.7}{2}
\twoe{0}{1}
\twoed{0}{2}
\end{tikzpicture}
\gS{-}
\begin{tikzpicture}[baseline=(A),outer sep=0pt,inner sep=0pt]
\node (A) at (0,0) {};
\foureu{0}{0}{b}
\twoeABS{0}{0.7}{2}
\foureu{0.7}{0}{c}
\end{tikzpicture}
\begin{tikzpicture}[baseline=(A),outer sep=0pt,inner sep=0pt]
\node (A) at (0,0) {};
\foureu{0.7}{0}{a}
\twoeABS{0}{0.7}{1}
\twoe{0}{1}
\twoed{0}{2}
\end{tikzpicture}
\gS{.}
\end{center}
Finally applying the Pl\"ucker relation to the vertical edge of color 2 in the first graph on the right and on the two edges of same color in the other ones, we get the syzygy
$$
\mathfrak{i}_{ab1}\mathfrak{j}_{12c}-\mathfrak{i}_{ac1}\mathfrak{j}_{12b}+\mathfrak{i}_{bc1}\mathfrak{j}_{12a}+\mathfrak{i}_{ab2}\mathfrak{j}_{21c}-\mathfrak{i}_{ac2}\mathfrak{j}_{21b}+\mathfrak{i}_{bc2}\mathfrak{j}_{21a}.
$$
The Hilbert series of the ideal generated by the resulting four syzygies for all combinations of $a,b,c$ coincides with that of $A_1$, so we are done with the case $p=2$.

In the case $p=3$, since the homological dimension is two, the c.i. property follows by~\cite[Le. 5.1]{shmel}, concrete syzygies can be obtained in the same way as for $p\geq 4$ in the following.  So consider $p\geq 4$. Here in the decomposition of $A_1$ from above, apart from the first three algebras, we have one factor for each $k=0,\ldots,p-3$ in the fourth and fifth one. In the resulting Hilbert series, the factors $(1-t_i^pt_j^2s_as_bs_c)$ in the denominator and numerator cancel out and in the reduced fraction, again all denominator factors correspond to generators. So we need syzygies corresponding to $(1-t_1^{k}t_2^{2p+3-k}s_1s_2s_3s_4)$, $k=4,\ldots,2p-1$. One of $k$ and $2p+3-k$ is always greater than $p+1$. Let us assume that $k > p+1$ in the following. Consider the graphs 
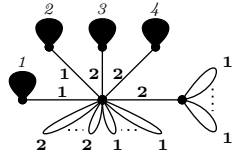
\begin{center}
\begin{tikzpicture}[baseline=(A),outer sep=0pt,inner sep=0pt]

\node (A) at (0,0) {};

\twoeARB{-1.05}{0}{0}{0}{1}{above=1pt};
\twoeARB{1.05}{0}{0}{0}{2}{above=1pt};
\twoeARB{-0.7}{0.7}{0}{0}{1}{left=2pt};
\twoeARB{0}{0.7}{0}{0}{2}{left=1pt};
\twoeARB{0.7}{0.7}{0}{0}{2}{left=2pt};
\staru{-1.05}{0}{1}
\staru{-0.7}{0.7}{2}
\staru{0}{0.7}{3}
\staru{0.7}{0.7}{4}

\begin{scope}[rotate around={180:(0,0)}]
\twostr{-0.2}{0.2}{1};
\draw[line width=0.7pt, draw=black,dotted]  (-0.45,0.4) -- (-0.25,0.4) ;
\twostr{-1}{0.4}{};
\node[font=\tiny,align=center] at (-0.8,0.6) {$\iiE{1}$};

\twostr{0.2}{0.2}{2};
\draw[line width=0.7pt, draw=black,dotted]  (0.45,0.4) -- (0.25,0.4) ;
\twostr{1}{0.4}{};
\node[font=\tiny,align=center] at (0.8,0.6) {$\iiE{2}$};

\end{scope}

\begin{scope}[shift={(1.05,0)}]
\begin{scope}[rotate around={-90:(0,0)}]
\twostr{-0.5}{0.4}{1};
\draw[line width=0.7pt, draw=black,dotted]  (-0.15,0.4) -- (0.15,0.4) ;

\twostr{0.5}{0.4}{1};

\end{scope}
\end{scope}
\end{tikzpicture}
\end{center}
These graphs are decomposable in \emph{two ways}. The first one is applying the Pl\"ucker relation on all two-edges of color $2$ but those two going up. In the resulting sum, all graphs are either disconnected or evaluate to zero, since they have a vertex with $p$ looping two-edges of color $1$ and one additional non-looping two-edge of color $1$. For the second one, we observe that we can move a looping two-edge of color $1$ from the rightmost to the central vertex by first applying the Pl\"ucker relation to all edges connected to the rightmost vertex - now the central and rightmost vertex are connected by a two-edge of color $1$ - and then applying it to this connecting edge and all looping  edges of color $1$ at the central vertex. Iterating this procedure gives us a sum of disconnected graphs plus the graph
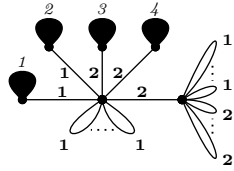
\begin{center}
\begin{tikzpicture}[baseline=(A),outer sep=0pt,inner sep=0pt]

\node (A) at (0,0) {};

\twoeARB{-1.05}{0}{0}{0}{1}{above=1pt};
\twoeARB{1.05}{0}{0}{0}{2}{above=1pt};
\twoeARB{-0.7}{0.7}{0}{0}{1}{left=2pt};
\twoeARB{0}{0.7}{0}{0}{2}{left=1pt};
\twoeARB{0.7}{0.7}{0}{0}{2}{left=2pt};
\staru{-1.05}{0}{1}
\staru{-0.7}{0.7}{2}
\staru{0}{0.7}{3}
\staru{0.7}{0.7}{4}

\begin{scope}[shift={(1.05,0)}]
\begin{scope}[rotate around={-90:(0,0)}]
\twostr{-0.2}{0.2}{1};
\draw[line width=0.7pt, draw=black,dotted]  (-0.45,0.4) -- (-0.25,0.4) ;
\twostr{-1}{0.4}{};
\node[font=\tiny,align=center] at (-0.8,0.6) {$\iiE{1}$};

\twostr{0.2}{0.2}{2};
\draw[line width=0.7pt, draw=black,dotted]  (0.45,0.4) -- (0.25,0.4) ;
\twostr{1}{0.4}{};
\node[font=\tiny,align=center] at (0.8,0.6) {$\iiE{2}$};

\end{scope}
\end{scope}

\begin{scope}[shift={(0,0)}]
\begin{scope}[rotate around={180:(0,0)}]
\twostr{-0.5}{0.4}{1};
\draw[line width=0.7pt, draw=black,dotted]  (-0.15,0.4) -- (0.15,0.4) ;

\twostr{0.5}{0.4}{1};

\end{scope}
\end{scope}
\end{tikzpicture}
\end{center}
Now finally applying in this graph the Pl\"ucker relation to all two-edges of color $1$ looping at the central vertex and the one that connects to the $(n-1)$-edge of color $1$, we see that it is decomposable as well. The resulting relation has the form
$$
\mathfrak{f}_{k}:=
\left(
\sum_{i=1}^{4} \sum_{
\begin{smallmatrix}
a+c=k \\
 b+d=2p+3-k 
\end{smallmatrix}} 
\mathfrak{j}_{abi} \mathfrak{k}_{cdi}
\right)
+
\left(
\sum_{(ij,l,m)\vdash(1234)}
 \sum_{
\begin{smallmatrix}
e=1,2\\
 a+c=k-\delta_{e1} \\
 b+d=2p+3-k-\delta_{e2}
\end{smallmatrix}} 
\mathfrak{i}_{ije}\mathfrak{j}_{abl} \mathfrak{j}_{cdm}
\right),
$$
where we ignore coefficients as they are not important for what follows. The notation $(ij,k,l)\vdash(1234)$ means that we sum over all subdivisions of the word $1234$ in two words of length one and one of length two, i.e. we have twelve summands. 

The same procedure  applies for $2p+3-k > p+1$ by interchanging two-edges of colors $1$ and $2$, yielding respective syzygies in this case as well.

What remains is to show that the $\mathfrak{f}_k$ are a Gr\"obner basis for the ideal of syzygies of $\CC[W_1]^{\SL_n}$. Consider the $\ZZ^{4n}$ grading of $\CC[W_1]^{\SL_n}$ given by sending $\mathfrak{i}_{abc}$, $\mathfrak{j}_{abc}$, $\mathfrak{k}_{abc}$ to respective basis elements $e^\mathfrak{i}_{abc}$, $e^\mathfrak{j}_{abc}$, $e^\mathfrak{k}_{abc}$. Denote the dual basis of $(\ZZ^{4n})^*$ by $x^\mathfrak{i}_{abc}$, $x^\mathfrak{j}_{abc}$, $x^\mathfrak{k}_{abc}$. 

Now we need to find a monomial grading so that the leading monomials of each $\mathfrak{f}_k$ have no variable in common. 
Let us construct a matrix $\mathcal{M}$ with $4n$ columns and of rank $4n$, so that the associated grading meets our requirements.
Our first goal is that the monomials $\mathfrak{j}_{abi} \mathfrak{k}_{cdi}$ are greater than the $\mathfrak{i}_{ije}\mathfrak{j}_{abk} \mathfrak{j}_{cdl}$. This is achieved by setting the first row $\mathcal{M}_1$ of $\mathcal{M}$ to
$$
\mathcal{M}_1:=\sum x^\mathfrak{i}_{abc} + \sum x^\mathfrak{j}_{abc} + \sum 3 x^\mathfrak{k}_{abc}.
$$
Now we present an algorithm determining a monomial grading satisfying our needs.

\noindent\rule[0.5ex]{\linewidth}{1pt}
\begin{algorithm}[\emph{Crosshair-sieve}]\label{algCS}
We see each row of the matrix $\mathcal{M}$ as a sieve (of increasing fineness), that  filters out some of the monomials we want to be the leading ones of the $\mathfrak{f}_k$.

Now for some row $\mathcal{M}_\nu$   consider a $(p-1)\times(p-4)\times 4$-matrix (or tensor if you want) $\mathcal{S}_\nu$, where in the $(r,s,t)$-th entry stands the degree of  $\mathfrak{j}_{r(p+1-r)t}\mathfrak{k}_{(s+2)(p-s)t}$. We say that the first entry gives the \emph{row}, the second one the \textit{column} and the third one the \textit{level} of the \emph{sieve} $\mathcal{S}_\nu$. So entries in the same level and row correspond to monomials sharing some $\mathfrak{j}_{r(p+1-r)t}$ and such in the same level and column some $\mathfrak{k}_{(s+2)(p-s)t}$ respectively.
Observe that the degrees of all monomials of some $\mathfrak{f}_k$ stand in the counterdiagonal of $\mathcal{S}_\nu$ given by $r+s+2=k$ (remember $k=4,\ldots,2p-1$).

We begin with $\mathcal{M}_\nu=0$. Now if we want a monomial $\mathfrak{j}_{abi} \mathfrak{k}_{cdi}$ to be filtered out (i.e. to be the one of highest degree of some $\mathfrak{f}_{a+c}$), we target this monomial setting 
$$
\mathcal{M}_\nu:=\mathcal{M}_\nu + x^\mathfrak{j}_{abi} + x^\mathfrak{k}_{cdi}.
$$
At the $t$-th level of $\mathcal{S}_\nu$, this looks like a crosshair. Exemplarily for the case  $p=9$, we have this picture:
\begin{center}
\begin{tikzpicture}
\matrix[matrix of math nodes,  nodes={draw,minimum size=5mm},row sep=-\pgflinewidth,column sep=-\pgflinewidth]
{
0 & 0 & 0 & 1 & 0 \\
0 & 0 & 0 & 1 & 0 \\
1 & 1 & 1 & 2 & 1 \\
0 & 0 & 0 & 1 & 0 \\
0 & 0 & 0 & 1 & 0 \\
0 & 0 & 0 & 1 & 0 \\
0 & 0 & 0 & 1 & 0 \\
0 & 0 & 0 & 1 & 0 \\
};
\end{tikzpicture}
\end{center}
Now $\mathfrak{j}_{abi} \mathfrak{k}_{cdi}$ is the leading monomial - with degree two - of $\mathfrak{f}_{a+c}$, but we have many leading monomials (of degree one) of other $\mathfrak{f}_k$'s that share a variable, which we do not want. So we have to target all other desired leading monomials in the same way. If this happens at another level of $\mathcal{S}_\nu$, all is fine, but if two desired leading monomials are at the same level - and this must happen for $p \geq 3$, since we only have four levels but need $2p-4$ leading monomials - we get unwanted crossings, as for example the italic ones in:
\begin{center}
\begin{tikzpicture}
\matrix[matrix of math nodes,  nodes={draw,minimum size=5mm},row sep=-\pgflinewidth,column sep=-\pgflinewidth]
{
1 & 0 & 0 & 1 & 0 \\
1 & 0 & 0 & 1 & 0 \\
\mathit{2} & 1 & 1 & 2 & 1 \\
1 & 0 & 0 & 1 & 0 \\
2 & 1 & 1 & \mathit{2} & 1 \\
1 & 0 & 0 & 1 & 0 \\
1 & 0 & 0 & 1 & 0 \\
1 & 0 & 0 & 1 & 0 \\
};
\end{tikzpicture}
\end{center}
So it may happen that in an $\mathfrak{f}_k$ some unwanted monomials are of the same degree as the desired leading polynomials with respect to $\mathcal{M}_\nu$, so the sieve $\mathcal{S}_\nu$ is \emph{too coarse}. But as long as it filters out at least one desired leading monomial, we do not have to target this one in the following sieve $\mathcal{S}_{\nu+1}$, so we get lesser unwanted crossings and may be able to filter out another desired leading monomial. \emph{So if at each stage we can filter out at least one, the algorithm will terminate with the desired set of leading monomials}. If the resulting matrix is not of full rank, we add rows to achieve full rank and found the desired  grading.
 
\end{algorithm}

\noindent\rule[0.5ex]{\linewidth}{1pt}

The termination of the algorithm obviously depends on the choice of the desired leading monomials, which results in a choice of arrangements of crosshairs eventually leading to a sieve that does not filter any desired leading monomial.

So we need to find a set of monomials for which the algorithm terminates. We do this by induction in $p$. Let $\mathcal{M}^p$ be the grading matrix for respective $p$ with first row $\mathcal{M}_1^p:=\sum x^\mathfrak{i}_{abc} + \sum x^\mathfrak{j}_{abc} + \sum 3 x^\mathfrak{k}_{abc}$ as defined above. Let $\mathcal{S}^p$ be the associated collection of sieves with $\mathcal{S}^p_i$ corresponding to $\mathcal{M}^p_i$. For $p=4$, apart from permuting the last index, we only have one choice for the leading monomials $\lm$:
$$
\lm(\mathfrak{f}_4)=\mathfrak{j}_{1p1}\mathfrak{k}_{3(p-1)1},
\quad
\lm(\mathfrak{f}_5)=\mathfrak{j}_{2(p-1)2}\mathfrak{k}_{3(p-1)2},
\quad
\lm(\mathfrak{f}_6)=\mathfrak{j}_{3(p-2)3}\mathfrak{k}_{3(p-1)3},
$$
$$
\lm(\mathfrak{f}_7)=\mathfrak{j}_{4(p-3)4}\mathfrak{k}_{3(p-1)4}.
$$
Applying Algorithm~\ref{algCS} to this choice, the sieve $\mathcal{S}^4_2$ becomes
\begin{center}
\begin{tikzpicture}
\node[matrix,  nodes={fill=white,draw,minimum size=3mm},row sep=-\pgflinewidth,column sep=-\pgflinewidth] (Leva) at (0,0)
{
\node(a1) {$2$}; \\
\node(a2) {$1$}; \\
\node(a3) {$1$}; \\
\node(a4) {$1$}; \\
};

\node[matrix,  nodes={fill=white,draw,minimum size=3mm},row sep=-\pgflinewidth,column sep=-\pgflinewidth] (Leva) at (0.5,0.1)
{
\node(b1) {$1$}; \\
\node(b2) {$2$}; \\
\node(b3) {$1$}; \\
\node(b4) {$1$}; \\
};

\node[matrix,  nodes={fill=white,draw,minimum size=3mm},row sep=-\pgflinewidth,column sep=-\pgflinewidth] (Leva) at (1,0.2)
{
\node(c1) {$1$}; \\
\node(c2) {$1$}; \\
\node(c3) {$2$}; \\
\node(c4) {$1$}; \\
};
\node[matrix,  nodes={fill=white,draw,minimum size=3mm},row sep=-\pgflinewidth,column sep=-\pgflinewidth] (Leva) at (1.5,0.3)
{
\node(d1) {$1$}; \\
\node(d2) {$1$}; \\
\node(d3) {$1$}; \\
\node(d4) {$2$}; \\
};
\draw[dashed,opacity=0.9] (a1.north east) -- (d1.north east);
\draw[dashed,opacity=0.9] (a1.north west) -- (d1.north west);
\draw[dashed,opacity=0.9] (a4.south east) -- (d4.south east);
\begin{pgfonlayer}{background}
\draw[dashed, opacity=0.9] (a4.south west) -- (d4.south west);
\end{pgfonlayer}
\end{tikzpicture}
\end{center}
This sieve already filters out our desired monomials.
We describe the first iteration step in detail: for $p=5$, the sieves have one additional column and row, so they are of the form
\begin{center}
\begin{tikzpicture}
\node[matrix,  nodes={fill=white,draw,minimum size=2mm},row sep=-\pgflinewidth,column sep=-\pgflinewidth] (Leva) at (0,0)
{
\node(a1) {}; & \node[fill=lightgray] (a12) {}; \\
\node(a2) {}; & \node[fill=lightgray] {}; \\
\node(a3) {};  & \node[fill=lightgray] {}; \\
\node(a4) {}; & \node[fill=lightgray] {}; \\
\node[fill=lightgray](a5) {}; & \node[fill=lightgray] (a52) {}; \\
};

\node[matrix,  nodes={fill=white,draw,minimum size=2mm},row sep=-\pgflinewidth,column sep=-\pgflinewidth] (Leva) at (1,0.2)
{
\node(b1) {}; & \node[fill=lightgray] {}; \\
\node(b2) {}; & \node[fill=lightgray] {}; \\
\node(b3) {}; & \node[fill=lightgray] {}; \\
\node(b4) {}; & \node[fill=lightgray] {}; \\
\node[fill=lightgray](b5) {}; & \node[fill=lightgray] {}; \\
};

\node[matrix,  nodes={fill=white,draw,minimum size=2mm},row sep=-\pgflinewidth,column sep=-\pgflinewidth] (Leva) at (2,0.4)
{
\node(c1) {}; & \node[fill=lightgray] {}; \\
\node(c2) {}; & \node[fill=lightgray] {}; \\
\node(c3) {}; & \node[fill=lightgray] {}; \\
\node(c4) {}; & \node[fill=lightgray] {}; \\
\node[fill=lightgray](c5) {}; & \node[fill=lightgray] {}; \\
};
\node[matrix,  nodes={fill=white,draw,minimum size=2mm},row sep=-\pgflinewidth,column sep=-\pgflinewidth] (Leva) at (3,0.6)
{
\node(d1) {}; & \node[fill=lightgray](d12) {}; \\
\node(d2) {}; & \node[fill=lightgray] {}; \\
\node(d3) {}; & \node[fill=lightgray] {}; \\
\node(d4) {}; & \node[fill=lightgray] {}; \\
\node[fill=lightgray](d5) {}; & \node[fill=lightgray](d52) {}; \\
};
\draw[dashed,opacity=0.9] (a12.north east) -- (d12.north east);
\draw[dashed,opacity=0.9] (a1.north west) -- (d1.north west);
\draw[dashed,opacity=0.9] (a52.south east) -- (d52.south east);
\begin{pgfonlayer}{background}
\draw[dashed, opacity=0.9] (a5.south west) -- (d5.south west);
\end{pgfonlayer}
\end{tikzpicture}
\end{center}
Observe that all entries on the counterdiagonal planes  correspond to monomials of one $\mathfrak{f}_k$. 
For $\mathfrak{f}_4,\ldots,\mathfrak{f}_7$, we keep the choice of leading monomials, while for $\mathfrak{f}_8$ and $\mathfrak{f}_9$ we choose
$$
\lm(\mathfrak{f}_8)=\mathfrak{j}_{4(p-3)1}\mathfrak{k}_{4(p-2)1},
\quad
\lm(\mathfrak{f}_9)=\mathfrak{j}_{5(p-4)2}\mathfrak{k}_{4(p-2)2},
$$
which results in $\mathcal{S}^5_2$ (with italic unwanted highest degrees) becoming
\begin{center}
\begin{tikzpicture}
\node[matrix,  nodes={fill=white,draw,minimum size=5mm},row sep=-\pgflinewidth,column sep=-\pgflinewidth] (Leva) at (0,0)
{
\node(a1) {$2$}; & \node (a12) {$\mathit{2}$}; \\
\node(a2) {$1$}; & \node {$1$}; \\
\node(a3) {$1$};  & \node {$1$}; \\
\node(a4) {$\mathit{2}$}; & \node {$2$}; \\
\node(a5) {$1$}; & \node (a52) {$1$}; \\
};

\node[matrix,  nodes={fill=white,draw,minimum size=5mm},row sep=-\pgflinewidth,column sep=-\pgflinewidth] (Leva) at (1.2,0.2)
{
\node(b1) {$1$}; & \node {$1$}; \\
\node(b2) {$2$}; & \node {$\mathit{2}$}; \\
\node(b3) {$1$}; & \node {$1$}; \\
\node(b4) {$1$}; & \node {$1$}; \\
\node(b5) {$\mathit{2}$}; & \node {$2$}; \\
};

\node[matrix,  nodes={fill=white,draw,minimum size=5mm},row sep=-\pgflinewidth,column sep=-\pgflinewidth] (Leva) at (2.4,0.4)
{
\node(c1) {$1$}; & \node {$0$}; \\
\node(c2) {$1$}; & \node {$0$}; \\
\node(c3) {$2$}; & \node {$1$}; \\
\node(c4) {$1$}; & \node {$0$}; \\
\node(c5) {$1$}; & \node {$0$}; \\
};
\node[matrix,  nodes={fill=white,draw,minimum size=5mm},row sep=-\pgflinewidth,column sep=-\pgflinewidth] (Leva) at (3.6,0.6)
{
\node(d1) {$1$}; & \node(d12) {$0$}; \\
\node(d2) {$1$}; & \node {$0$}; \\
\node(d3) {$1$}; & \node {$0$}; \\
\node(d4) {$2$}; & \node {$1$}; \\
\node(d5) {$1$}; & \node(d52) {$0$}; \\
};
\draw[dashed,opacity=0.9] (a12.north east) -- (d12.north east);
\draw[dashed,opacity=0.9] (a1.north west) -- (d1.north west);
\draw[dashed,opacity=0.9] (a52.south east) -- (d52.south east);
\begin{pgfonlayer}{background}
\draw[dashed, opacity=0.9] (a5.south west) -- (d5.south west);
\end{pgfonlayer}
\end{tikzpicture}
\end{center}
We see that $\lm(\mathfrak{f}_9)=\mathfrak{j}_{5(p-4)2}\mathfrak{k}_{4(p-2)2}$ and $\lm(\mathfrak{f}_4)=\mathfrak{j}_{1p1}\mathfrak{k}_{3(p-1)1}$ are filtered out and so in the next sieve $\mathcal{S}^5_3$ we do not have to target them, which results in
\begin{center}
\begin{tikzpicture}
\node[matrix,  nodes={fill=white,draw,minimum size=5mm},row sep=-\pgflinewidth,column sep=-\pgflinewidth] (Leva) at (0,0)
{
\node(a1) {$0$}; & \node (a12) {$1$}; \\
\node(a2) {$0$}; & \node {$1$}; \\
\node(a3) {$0$};  & \node {$1$}; \\
\node(a4) {$1$}; & \node {$2$}; \\
\node(a5) {$0$}; & \node (a52) {$1$}; \\
};

\node[matrix,  nodes={fill=white,draw,minimum size=5mm},row sep=-\pgflinewidth,column sep=-\pgflinewidth] (Leva) at (1.2,0.2)
{
\node(b1) {$1$}; & \node {$0$}; \\
\node(b2) {$2$}; & \node {$1$}; \\
\node(b3) {$1$}; & \node {$0$}; \\
\node(b4) {$1$}; & \node {$0$}; \\
\node(b5) {$1$}; & \node {$0$}; \\
};

\node[matrix,  nodes={fill=white,draw,minimum size=5mm},row sep=-\pgflinewidth,column sep=-\pgflinewidth] (Leva) at (2.4,0.4)
{
\node(c1) {$1$}; & \node {$0$}; \\
\node(c2) {$1$}; & \node {$0$}; \\
\node(c3) {$2$}; & \node {$1$}; \\
\node(c4) {$1$}; & \node {$0$}; \\
\node(c5) {$1$}; & \node {$0$}; \\
};
\node[matrix,  nodes={fill=white,draw,minimum size=5mm},row sep=-\pgflinewidth,column sep=-\pgflinewidth] (Leva) at (3.6,0.6)
{
\node(d1) {$1$}; & \node(d12) {$0$}; \\
\node(d2) {$1$}; & \node {$0$}; \\
\node(d3) {$1$}; & \node {$0$}; \\
\node(d4) {$2$}; & \node {$1$}; \\
\node(d5) {$1$}; & \node(d52) {$0$}; \\
};
\draw[dashed,opacity=0.9] (a12.north east) -- (d12.north east);
\draw[dashed,opacity=0.9] (a1.north west) -- (d1.north west);
\draw[dashed,opacity=0.9] (a52.south east) -- (d52.south east);
\begin{pgfonlayer}{background}
\draw[dashed, opacity=0.9] (a5.south west) -- (d5.south west);
\end{pgfonlayer}
\end{tikzpicture}
\end{center}
So $\mathcal{S}^5_3$ filters out the remaining leading monomials and we are done.
Now for general $p$, choose the leading monomials for $\mathfrak{f}_4,\ldots,\mathfrak{f}_{2p-3}$ as for $p-1$ and 
$$
\lm(\mathfrak{f}_{2p-2})=\mathfrak{j}_{(p-1)21}\mathfrak{k}_{(p-1)31},
\quad
\lm(\mathfrak{f}_{2p-1})=\mathfrak{j}_{p12}\mathfrak{k}_{(p-1)32}.
$$
With this choice, the desired leading monomial of $\mathfrak{f}_{2p-1}$ (among others, but this is just a bonus) is obviously filtered out by $\mathcal{S}_2^p$.
In $\mathcal{S}_3^p$, it must not be targeted any more, and since the corresponding entry is the only one right and below the one of $\lm(\mathfrak{f}_{2p-2})$, at least this will now be filtered out by $\mathcal{S}_3^p$. So every $\mathcal{S}_\nu^p$ will filter out at least $\lm(\mathfrak{f}_{2p-\nu+1})$, which means that we are done. The $\mathfrak{f}_k$ with respective leading monomials are a Gr\"obner basis of the ideal of relations of $\CC[W_1]^{\SL_n}$ and thus it is a complete intersection.

 \end{proof}

\begin{remark}
In fact, we only need one row of $\mathcal{M}$ to achieve our desired leading monomials by subsequently scaling the $\mathcal{M}_\nu$ with highest index $\nu$ by a factor smaller than one and adding it to $\mathcal{M}_{\nu-1}$. With factor $\frac{1}{2}$, in the case $p=5$ from above, the one remaining sieve then becomes
\begin{center}
\begin{tikzpicture}
\node[matrix,  nodes={fill=white,draw,minimum size=5.5mm},row sep=-\pgflinewidth,column sep=-\pgflinewidth] (Leva) at (0,0)
{
\node(a1) {$2$}; & \node (a12) {${\scriptstyle\frac{5}{2}}$}; \\
\node(a2) {$1$}; & \node {${\scriptstyle\frac{3}{2}}$}; \\
\node(a3) {$1$};  & \node {${\scriptstyle\frac{3}{2}}$}; \\
\node(a4) {${\scriptstyle\frac{5}{2}}$}; & \node {$3$}; \\
\node(a5) {$1$}; & \node (a52) {${\scriptstyle\frac{3}{2}}$}; \\
};

\node[matrix,  nodes={fill=white,draw,minimum size=5.5mm},row sep=-\pgflinewidth,column sep=-\pgflinewidth] (Leva) at (1.2,0.2)
{
\node(b1) {${\scriptstyle\frac{3}{2}}$}; & \node {$1$}; \\
\node(b2) {$3$}; & \node {${\scriptstyle\frac{5}{2}}$}; \\
\node(b3) {${\scriptstyle\frac{3}{2}}$}; & \node {$1$}; \\
\node(b4) {${\scriptstyle\frac{3}{2}}$}; & \node {$1$}; \\
\node(b5) {${\scriptstyle\frac{5}{2}}$}; & \node {$2$}; \\
};

\node[matrix,  nodes={fill=white,draw,minimum size=5.5mm},row sep=-\pgflinewidth,column sep=-\pgflinewidth] (Leva) at (2.4,0.4)
{
\node(c1) {${\scriptstyle\frac{3}{2}}$}; & \node {$0$}; \\
\node(c2) {${\scriptstyle\frac{3}{2}}$}; & \node {$0$}; \\
\node(c3) {$3$}; & \node {${\scriptstyle\frac{3}{2}}$}; \\
\node(c4) {${\scriptstyle\frac{3}{2}}$}; & \node {$0$}; \\
\node(c5) {${\scriptstyle\frac{3}{2}}$}; & \node {$0$}; \\
};
\node[matrix,  nodes={fill=white,draw,minimum size=5.5mm},row sep=-\pgflinewidth,column sep=-\pgflinewidth] (Leva) at (3.6,0.6)
{
\node(d1) {${\scriptstyle\frac{3}{2}}$}; & \node(d12) {$0$}; \\
\node(d2) {${\scriptstyle\frac{3}{2}}$}; & \node {$0$}; \\
\node(d3) {${\scriptstyle\frac{3}{2}}$}; & \node {$0$}; \\
\node(d4) {$3$}; & \node {${\scriptstyle\frac{3}{2}}$}; \\
\node(d5) {${\scriptstyle\frac{3}{2}}$}; & \node(d52) {$0$}; \\
};
\draw[dashed,opacity=0.9] (a12.north east) -- (d12.north east);
\draw[dashed,opacity=0.9] (a1.north west) -- (d1.north west);
\draw[dashed,opacity=0.9] (a52.south east) -- (d52.south east);
\begin{pgfonlayer}{background}
\draw[dashed, opacity=0.9] (a5.south west) -- (d5.south west);
\end{pgfonlayer}
\end{tikzpicture}
\end{center}
The grading matrix may then be filled up arbitrarily. Of course this gives a different grading but the leading monomials stay the same.
\end{remark}

\subsection{The case $(\SL_{2p+1},W_2)$}
 
 \begin{proposition}\label{prop:W2}
The ring of invariants $\CC[W_2]^{\SL_{2p+1}}$ is a complete intersection. Its homological dimension is two for $p=2$ and $2p-3$ for $p\geq 3$. It is minimally generated by the 
$$
\mathfrak{i}_{abc}, ab \in \{12,13,23\}, c \in \{1,2\}; 
\quad
\mathfrak{j}_{abc}, a,b \geq 1, c \in \{1,2,3\};
\quad
\mathfrak{k}_{ab4}, a,b\geq 3.
$$ 
from Proposition~\ref{prop:W1} and in addition 
\begin{center}
\gS{\mathfrak{h}_{a}=}
\begin{tikzpicture}[baseline=(A),outer sep=0pt,inner sep=0pt]
\node (A) at (0,0) {};
\staru{0}{0}{a};
\onee{0}{1};
\end{tikzpicture}
\kern-0.1em
\gS{a \in \{1,2,3\};}
\
\gS{\mathfrak{l}_{ab}=}
\begin{tikzpicture}[baseline=(A),outer sep=0pt,inner sep=0pt]
\clip (-0.9,-1.3) rectangle (0.9,1.1);
\node (A) at (0,0) {};

\twostr{-0.2}{0.2}{1};
\draw[line width=0.7pt, draw=black,dotted]  (-0.45,0.4) -- (-0.25,0.4) ;
\twostr{-1}{0.4}{};
\node[font=\tiny,align=center] at (-0.8,0.6) {$\iiE{1}$};
\draw[decorate, decoration={brace}, line width=0.5pt, draw=black]  (-0.8,0.75) -- (-0.2,0.75) node[font=\tiny,midway, above=3pt]  {$a$};

\twostr{0.2}{0.2}{2};
\draw[line width=0.7pt, draw=black,dotted]  (0.45,0.4) -- (0.25,0.4) ;
\twostr{1}{0.4}{};
\node[font=\tiny,align=center] at (0.8,0.6) {$\iiE{2}$};
\draw[decorate, decoration={brace,mirror}, line width=0.5pt, draw=black]  (0.8,0.75) -- (0.2,0.75) node[font=\tiny,midway, above=3pt]  {$b$};

\onee{0}{1};

\end{tikzpicture}
\kern-0.1em
\gS{, a,b\geq 0; }
\\
\gS{\mathfrak{m}_{abc}=}
\begin{tikzpicture}[baseline=(A),outer sep=0pt,inner sep=0pt]
\clip (-0.9,-1.3) rectangle (1,1.1);
\node (A) at (0,0) {};

\twostr{-0.2}{0.2}{1};
\draw[line width=0.7pt, draw=black,dotted]  (-0.45,0.4) -- (-0.25,0.4) ;
\twostr{-1}{0.4}{};
\node[font=\tiny,align=center] at (-0.8,0.6) {$\iiE{1}$};
\draw[decorate, decoration={brace}, line width=0.5pt, draw=black]  (-0.8,0.75) -- (-0.2,0.75) node[font=\tiny,midway, above=3pt]  {$a$};

\twostr{0.2}{0.2}{2};
\draw[line width=0.7pt, draw=black,dotted]  (0.45,0.4) -- (0.25,0.4) ;
\twostr{1}{0.4}{};
\node[font=\tiny,align=center] at (0.8,0.6) {$\iiE{2}$};
\draw[decorate, decoration={brace,mirror}, line width=0.5pt, draw=black]  (0.8,0.75) -- (0.2,0.75) node[font=\tiny,midway, above=3pt]  {$b-2$};

\onee{0}{1};

\begin{scope}[rotate around={180:(-0.7,0)}]
\staru{-0.7}{0}{d};
\end{scope}
\twoeARB{-0.7}{0}{0}{0}{2}{below=2pt};
\begin{scope}[rotate around={180:(0.7,0)}]
\staru{0.7}{0}{e};
\end{scope}
\twoeARB{0.7}{0}{0}{0}{2}{below=2pt};
\end{tikzpicture}
\gS{, a,b\geq 2,  de \in \{23,13,12\}, c=\{1,2,3\}\setminus\{d,e\}. }
\end{center}

\end{proposition}
 
 \begin{lemma}\label{le:W2}
 Denote by $x$ the images of the $c$ in $\CC[V_1]^U$ and by $y$ those in $\CC[V_2]^{U^o}$.  The algebra  $A_2=(\CC[V_1]^U\otimes\CC[V_2]^{U^o})^T$ is minimally generated by $x^*_2, y^\circ_{p+1}, y^\Diamond_1$ and the entries of the matrices
 \setlength{\arraycolsep}{0.5mm}
   $$
 d:=
 \begin{bmatrix}
 y^\circ_p \\
 y^{(1)}_{p+1} 
 \end{bmatrix}
 \cdot
 \begin{bmatrix}
 x_1 & x^*_3
 \end{bmatrix}
 ,
 \quad
 e:=
 \begin{bmatrix}
 y_p \\
 y^{(1)}_1
 \end{bmatrix}
 \cdot
 \begin{bmatrix}
 x^{(1)}_{2} & x^{(2)}_{2} 
 \end{bmatrix}
 ,
 $$
 $$
 f:=
  \begin{bmatrix}
 y^\circ_1 \\
 y^{(1)}_2
 \end{bmatrix}
\cdot
 \begin{bmatrix}
 x^{(1)}_1 & x^{(2)}_1 & x_p 
 \end{bmatrix}
 ,
 \quad
 g:=\begin{bmatrix}
 y_1 \\
 y^\Diamond_2
 \end{bmatrix}
 \cdot
 \begin{bmatrix}
 x^{(1)}_{p+1} & x^{(2)}_{p+1} & x_1^*
 \end{bmatrix}
 ,
 $$
  $$
h^{(k)}:=\begin{bmatrix}
x_{2+k} \\
 x^*_{4+k}
 \end{bmatrix}
 \cdot
 \begin{bmatrix}
 y^\circ_{p-k+1} 
 &
 y^{(1)}_{p-k}
 \end{bmatrix}
 ,
 \quad
 i^{(k)}:=\begin{bmatrix}
x^{(1)}_{p-k} \\
 x^{(2)}_{p-k} 
 \end{bmatrix}
 \cdot
 \begin{bmatrix}
 y_{2+k} & y^\Diamond_{3+k} 
 \end{bmatrix}
 ,
 \quad
  k=0,\ldots,p-3.
 $$
The ideal of syzygies is generated by all $2\times 2$-minors of the matrices $d$, $e$, $f$, $g$, $h^{(k)}$, $i^{(k)}$ and
$$
f_{a3}g_{b3}-p f_{a1} g_{b2} + p f_{a2} g_{b1}, \quad a,b \in \{1,2\}.
$$
 \end{lemma}

\begin{proof}
We proceed exactly as in Lemma~\ref{le:W1}.
\end{proof}

 \begin{proof}[Proof of Prop.~\ref{prop:W2}]
 The list of generators can be extracted from~\cite[Le. 2.5]{shmel}, it is minimal due to~\cite[Proof of Thm. 0.2]{shmel}.
 We come to the syzygies.
 Let us first have a look at the Hilbert series of $A_2$ with respect to the $\ZZ^6$-grading of $W_2$, where variables $t_i$, $s_i$ and $r$ correspond to copies of $\BigWedge^2V$, $\BigWedge^{n-1}V$ and $V$ respectively. Lemma~\ref{le:W2} provides us with generators and syzygies. We have 
\begin{align*}
 A_2 = 
\CC[x^*_2, y^\circ_{p+1},y^\Diamond_1]  &  \otimes  \CC[y^{(1)}_{p+1},y^\circ_p,y^*_1,x_1,x^*_3]^T \otimes  \CC[y^{(1)}_{1} ,y_p,x^{(1)}_{2},x^{(2)}_{2}]^T   \\
 & \otimes  \CC[y^{(1)}_{2}, y^\circ_1,x^{(1)}_1,x^{(2)}_1,x_p,y_1,y^\Diamond_2,x^{(1)}_{p+1},x^{(2)}_{p+1},x^*_1]^T
   \\
 & \otimes  \bigotimes_{k=0}^{p-3}\CC[x_{2+k},x^*_{4+k},y^\circ_{p-k+1},y^{(1)}_{p-k}]^T \\  
  & \otimes \bigotimes_{k=0}^{p-3}\CC[y_{2+k},y^\Diamond_{3+k},x^{(1)}_{p-k},x^{(2)}_{p-k}]^T.
    \end{align*}
The first four algebras have the Hilbert series
$${\scriptstyle
\frac{1}{(1-t_1s_1s_2)(1-t_2^pr)(1-s_3r)},
\quad
\frac{(1-t_1^3t_2^{2p-1}s_1s_2s_3r)}{(1-t_1t_2^{p-1}r)(1-t_1t_2^{p}s_3)(1-t_1^2t_2^{p-1}s_1s_2r)(1-t_1^2t_2^ps_1s_2s_3)},
}
$$
$$
{\scriptstyle
\frac{(1-t_1^2t_2^ps_1s_2s_3)}{(1-t_1t_2^ps_1)(1-t_1t_2^ps_2)(1-t_1s_1s_3)(1-t_1s_2s_3)},
}$$
$$
{\scriptstyle
\frac{(1-t_1^pt_2^2s_1s_2s_3)(1-t_1^pt_2s_1s_2r)}{(1-s_1r)(1-s_2r)(1-t_1^pr)(1-t_2s_1s_3)(1-t_2s_2s_3)(1-t_1^pt_2s_3)(1-t_1^pt_2s_1)(1-t_1^pt_2s_2)(1-t_2s_1s_2),}
}
$$
whereas the ones of the $k$-th factors of the fifth and sixth algebra are
$$
{\scriptstyle
\frac{(1-t_1^{2k+5}t_2^{2p-2k-3}s_1s_2s_3r)}{(1-t_1^{2+k}t_2^{p-k-1}s_3)(1-t_1^{3+k}t_2^{p-k-1}s_1s_2s_3)(1-t_1^{2+k}t_2^{p-k-2}r)(1-t_1^{3+k}t_2^{p-k-2}s_1s_2r)},
}
$$
$$
{\scriptstyle
\frac{(1-t_1^{2p-2k-2}t_2^{2k+4}s_1s_2s_3r)}{(1-t_1^{p-k-1}t_2^{2+k}s_1)(1-t_1^{p-k-1}t_2^{2+k}s_2)(1-t_1^{p-k-1}t_2^{2+k}s_1s_3r)(1-t_1^{p-k-1}t_2^{2+k}s_2s_3r)}.
}
$$
As in the proof of Proposition~\ref{prop:W1}, we distinguish between $p=2,3,\geq 4$. For $p=2$, we have homological dimension two and according to the Hilbert series and the given $18$ generators, there are two syzygies corresponding to the factors $(1-t_1^2t_2^2s_1s_2s_3)$ and $(1-t_1^3t_2^3s_1s_2s_3r)$. The first one already occured in $\CC[W_1]^{\SL_5}$, it is 
$$
\mathfrak{i}_{121}\mathfrak{j}_{123}-\mathfrak{i}_{131}\mathfrak{j}_{122}+\mathfrak{i}_{231}\mathfrak{j}_{121}+\mathfrak{i}_{122}\mathfrak{j}_{213}-\mathfrak{i}_{132}\mathfrak{j}_{212}+\mathfrak{i}_{232}\mathfrak{j}_{211}.
$$
For an explicit form of the second one, consider the graphs
\begin{center}
\begin{tikzpicture}[baseline=(A),outer sep=0pt,inner sep=0pt]

\node (A) at (0,0) {};

\twoeARB{0}{0.4}{0}{0}{1}{right=1pt};
\foureu{-0.7}{0}{1}
\foureu{0}{0.4}{2}
\foureu{0.7}{0}{3}
\twoeABS{0}{-0.7}{1}
\twoeABS{0}{0.7}{2}
\begin{scope}[rotate around={45:(0,0)}]
\onee{0}{1}
\end{scope}
\twoeARB{-0.35}{-0.7}{0}{0}{1}{left=3pt};
\begin{scope}[shift={(-0.35,-0.7)}]
\begin{scope}[rotate around={90:(0,0)}]
\twoe{0}{2}
\twoed{0}{2}
\end{scope}
\end{scope}
\end{tikzpicture}
\gS{=}
\gS{-2}
\begin{tikzpicture}[baseline=(A),outer sep=0pt,inner sep=0pt]

\node (A) at (0,0) {};

\twoeARB{0}{0.4}{0}{0}{1}{right=1pt};
\foureu{-0.7}{0}{1}
\foureu{0}{0.4}{2}
\foureu{0.7}{0}{3}
\twoeABS{0}{-0.7}{1}
\twoeABS{0}{0.7}{2}
\begin{scope}[rotate around={45:(0,0)}]
\onee{0}{1}
\end{scope}
\twoeARB{-0.35}{-0.7}{0}{0}{2}{left=3pt};
\begin{scope}[shift={(-0.35,-0.7)}]
\begin{scope}[rotate around={90:(0,0)}]
\twoe{0}{1}
\twoed{0}{2}
\end{scope}
\end{scope}
\end{tikzpicture}
\gS{.}
\end{center}
Applying the Pl\"ucker relation to the diagonal edges in each of both graphs gives a relation, which we can write - ignoring coefficients - as
$$
 \sum_{(i,j,k)\vdash(123) } 
\mathfrak{h}_{i}\mathfrak{j}_{12j} \mathfrak{j}_{21k}
+
 \sum_{(ij,k)\vdash(123) } \left(
\mathfrak{i}_{ij1}(\mathfrak{j}_{12k} \mathfrak{l}_{11} + \mathfrak{j}_{21k} \mathfrak{l}_{02})+\mathfrak{i}_{ij2}(\mathfrak{j}_{12k} \mathfrak{l}_{20} + \mathfrak{j}_{21k} \mathfrak{l}_{11})
\right).
$$

Now consider the case $p=3$. Here as with $W_1$, we  encounter all types of invariants but $\mathfrak{k}_{ab4}$. We need three relations corresponding to $(t_1^{k}t_2^{8-k}s_1s_2s_3r)$ for $k=3,4,5$. Consider the graphs
\begin{center}
\begin{tikzpicture}[baseline=(A),outer sep=0pt,inner sep=0pt]

\node (A) at (0,0) {};

\twoeARB{-1.05}{0}{0}{0}{2}{above=1pt};
\twoeARB{1.05}{0}{0}{0}{2}{above=1pt};
\twoeARB{-0.7}{0.7}{0}{0}{2}{left=2pt};
\twoeARB{0}{0.7}{0}{0}{2}{left=1pt};

\sixeu{-1.05}{0}{1}
\sixeu{-0.7}{0.7}{2}
\sixeu{0}{0.7}{3}

\begin{scope}[rotate around={135:(0,0)}]
\onee{0}{1}
\end{scope}

\twoed{0}{2}

\begin{scope}[shift={(1.05,0)}]
\begin{scope}[rotate around={-15:(0,0)}]
\twoed{0}{1}
\end{scope}
\begin{scope}[rotate around={90:(0,0)}]
\twoed{0}{1}
\end{scope}
\begin{scope}[rotate around={195:(0,0)}]
\twoed{0}{1}
\end{scope}
\end{scope}
\end{tikzpicture}
\gS{,}
\begin{tikzpicture}[baseline=(A),outer sep=0pt,inner sep=0pt]

\node (A) at (0,0) {};

\twoeARB{-1.05}{0}{0}{0}{1}{above=1pt};
\twoeARB{1.05}{0}{0}{0}{2}{above=1pt};
\twoeARB{-0.7}{0.7}{0}{0}{2}{left=2pt};
\twoeARB{0}{0.7}{0}{0}{2}{left=1pt};

\sixeu{-1.05}{0}{1}
\sixeu{-0.7}{0.7}{2}
\sixeu{0}{0.7}{3}

\begin{scope}[rotate around={135:(0,0)}]
\onee{0}{1}
\end{scope}

\twoed{0}{2}

\begin{scope}[shift={(1.05,0)}]
\begin{scope}[rotate around={-15:(0,0)}]
\twoed{0}{1}
\end{scope}
\begin{scope}[rotate around={90:(0,0)}]
\twoed{0}{1}
\end{scope}
\begin{scope}[rotate around={195:(0,0)}]
\twoed{0}{1}
\end{scope}
\end{scope}
\end{tikzpicture}
\gS{,}
\begin{tikzpicture}[baseline=(A),outer sep=0pt,inner sep=0pt]

\node (A) at (0,0) {};

\twoeARB{-1.05}{0}{0}{0}{1}{above=1pt};
\twoeARB{1.05}{0}{0}{0}{2}{above=1pt};
\twoeARB{-0.7}{0.7}{0}{0}{1}{left=2pt};
\twoeARB{0}{0.7}{0}{0}{2}{left=1pt};

\sixeu{-1.05}{0}{1}
\sixeu{-0.7}{0.7}{2}
\sixeu{0}{0.7}{3}

\begin{scope}[rotate around={135:(0,0)}]
\onee{0}{1}
\end{scope}

\twoed{0}{2}

\begin{scope}[shift={(1.05,0)}]
\begin{scope}[rotate around={-15:(0,0)}]
\twoed{0}{1}
\end{scope}
\begin{scope}[rotate around={90:(0,0)}]
\twoed{0}{1}
\end{scope}
\begin{scope}[rotate around={195:(0,0)}]
\twoed{0}{1}
\end{scope}
\end{scope}
\end{tikzpicture}
\gS{.}
\end{center}
Applying Pl\"ucker relations to them results - for $k=3,4,5$ - in
$$
\mathfrak{g}_k=
 \sum_{
  (i,j,l)\vdash(123)} 
  \sum
\mathfrak{h}_{i}\mathfrak{j}_{abj} \mathfrak{j}_{cdl}
+
\sum_{(ij,l)\vdash(123)}
\left( \sum 
\mathfrak{j}_{abl} \mathfrak{m}_{cdl}
+ \sum
\mathfrak{i}_{ije}\mathfrak{j}_{abl} \mathfrak{l}_{cd}
\right),
$$
where in the second and fourth sum, we sum over $a+c=k, b+d=2p+2-k$ and in the last one over $e=1,2$ and $ a+c=k-\delta_{e1}, b+d=2p+2-k-\delta_{e2}$. In the terms $\mathfrak{j}_{abl} \mathfrak{m}_{cdl}$, we observe that $a=b=2$ must hold, so $\mathfrak{g}_3, \mathfrak{g}_4, \mathfrak{g}_5$ contain terms  $\mathfrak{j}_{22l} \mathfrak{m}_{13l}, \mathfrak{j}_{22l} \mathfrak{m}_{22l}, \mathfrak{j}_{22l} \mathfrak{m}_{31l}$ respectively. We apply the Crosshair-sieve-algorithm~\ref{algCS}. First we make the $\mathfrak{m}$ the variables of highest degree and second, we filter out $\mathfrak{j}_{221} \mathfrak{m}_{131}, \mathfrak{j}_{222} \mathfrak{m}_{222}, \mathfrak{j}_{223} \mathfrak{m}_{313}$ as leading monomials of $\mathfrak{g}_3, \mathfrak{g}_4, \mathfrak{g}_5$, which shows that they form a Gr\"obner basis and that $\CC[W_2]^{\SL_7}$ is a complete intersection.

Now for $p\geq 4$, the situation is similar, but the relations $\mathfrak{g}_k$ now have additional terms:
$$
\mathfrak{g}_k=
\sum \mathfrak{k}_{ab4}\mathfrak{l}_{cd}
+
 \sum_{
  (i,j,l)\vdash(123)} 
\mathfrak{h}_{i}\mathfrak{j}_{abj} \mathfrak{j}_{cdl}
+
\sum_{(ij,l)\vdash(123)}
\left( \sum 
\mathfrak{j}_{abl} \mathfrak{m}_{cdl}
+ \sum
\mathfrak{i}_{ije}\mathfrak{j}_{abl} \mathfrak{l}_{cd}
\right).
$$
We apply the Crosshair-sieve-algorithm~\ref{algCS} - in a preliminary step $\mathcal{M}_1$ making variables $\mathfrak{m}$ and $\mathfrak{k}$ the ones of highest degree -  with the following sieves: 
$\mathcal{S}_\nu$ now comprises two parts -  a $p\times(p-2)\times 3$-matrix $\mathcal{S}_{\nu,1}$ and a $(p+1)\times(p-3)$-matrix $\mathcal{S}_{\nu,2}$ , where in the $(r,s,t)$-th entry of $\mathcal{S}_{\nu,1}$ stands the degree of  $\mathfrak{j}_{r(p+1-r)t}\mathfrak{m}_{(s+1)(p-s)t}$ and in the $(r,s)$-th entry of $\mathcal{S}_{\nu,2}$ the degree of $\mathfrak{k}_{(s+2)(p-s)4}\mathfrak{l}_{(r-1)(p+1-r)}$. So for $p=4$, the sieves are of the form
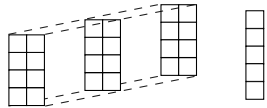
\begin{center}
\begin{tikzpicture}
\node[matrix,  nodes={fill=white,draw,minimum size=2mm},row sep=-\pgflinewidth,column sep=-\pgflinewidth] (Leva) at (0,0)
{
\node(a1) {}; & \node (a12) {}; \\
\node(a2) {}; & \node {}; \\
\node(a3) {};  & \node {}; \\
\node(a5) {}; & \node (a52) {}; \\
};

\node[matrix,  nodes={fill=white,draw,minimum size=2mm},row sep=-\pgflinewidth,column sep=-\pgflinewidth] (Leva) at (1,0.2)
{
\node(b1) {}; & \node {}; \\
\node(b2) {}; & \node {}; \\
\node(b3) {}; & \node {}; \\
\node(b5) {}; & \node {}; \\
};

\node[matrix,  nodes={fill=white,draw,minimum size=2mm},row sep=-\pgflinewidth,column sep=-\pgflinewidth] (Leva) at (2,0.4)
{
\node(c1) {}; & \node(c12) {}; \\
\node(c2) {}; & \node {}; \\
\node(c3) {}; & \node {}; \\
\node(c5) {}; & \node(c52) {}; \\
};
\node[matrix,  nodes={fill=white,draw,minimum size=2mm},row sep=-\pgflinewidth,column sep=-\pgflinewidth] (Leva) at (3,0.2)
{
\node(d1) {}; \\
\node(d2) {};  \\
\node(d3) {};  \\
\node(d4) {};  \\
\node(d5) {};  \\
};
\draw[dashed,opacity=0.9] (a12.north east) -- (c12.north east);
\draw[dashed,opacity=0.9] (a1.north west) -- (c1.north west);
\draw[dashed,opacity=0.9] (a52.south east) -- (c52.south east);
\begin{pgfonlayer}{background}
\draw[dashed, opacity=0.9] (a5.south west) -- (c5.south west);
\end{pgfonlayer}
\end{tikzpicture}
\end{center}
Again, monomials occuring in one $\mathfrak{g}_k$ lie on the counterdiagonals. For $\mathfrak{g}_3,\mathfrak{g}_4,\mathfrak{g}_5$, we choose the same leading monomials as in the case $p=3$, while for $\mathfrak{g}_6,\mathfrak{g}_7$, we take
$$
\lm(\mathfrak{g}_6)=\mathfrak{l}_{31}\mathfrak{k}_{334},
\quad
\lm(\mathfrak{g}_7)=\mathfrak{j}_{411}\mathfrak{m}_{321}.
$$
With this choice $\mathcal{S}_2$ filters out all leading monomials but the ones of $\mathfrak{g}_4$ and $\mathfrak{g}_6$, which is done by $\mathcal{S}_3$. If for general $p$, we choose the same leading monomials as for $p-1$ and for the two additional relations we take
$$
\lm(\mathfrak{g}_{2p})=\mathfrak{j}_{(p-1)22}\mathfrak{m}_{(p-3)42},
\quad
\lm(\mathfrak{g}_{2p-1})=\mathfrak{j}_{p33}\mathfrak{m}_{(p-3)43},
$$
by the same argument as in the proof of Proposition~\ref{prop:W1}, each  $\mathcal{S}^p_\nu$ filters out at least the leading monomial of the remaining $\mathfrak{g}_k$ with highest index, so the Crosshair-sieve-algorithm terminates, the $\mathfrak{g}_k$ are a Gr\"obner basis and $\CC[W_2]^{\SL_n}$ is a complete intersection.

\end{proof}

\subsection{The case $(\SL_{2p+1},W_3)$}
Let from now on all two-edges be of color one.

 \begin{proposition}\label{prop:W3}
The ring of invariants $\CC[W_3]^{\SL_{2p+1}}$ is a complete intersection. Its homological dimension is $2p-2$. It is minimally generated by $\mathfrak{i}_{121}$ from Proposition~\ref{prop:W1} and
\begin{center}
\gS{\mathfrak{n}_{ab}=}
\begin{tikzpicture}[baseline=(A),outer sep=0pt,inner sep=0pt]
\node (A) at (0,0) {};
\twoeS{0}{0.7}{}{0}
\staru{0}{0}{a};
\staru{0.7}{0}{b};
\end{tikzpicture}
\kern-0.3em
\gS{,ab \in \{11,12,22\},}
\gS{\mathfrak{o}=}
\kern-0.5em
\begin{tikzpicture}[baseline=(A),outer sep=0pt,inner sep=0pt]
\node (A) at (0,0) {};

\twoeS{0}{0.7}{}{0.35};
\twoeSd{0}{0.7}{}{0.35};
\draw[line width=0.7pt, draw=black,dotted]  (0.35,0.15) --(0.35,-0.15) ;

\draw[black, fill=black] (0,0) circle [radius=1.5pt];
\draw[black, fill=black] (0.7,0) circle [radius=1.5pt];
\end{tikzpicture}
\kern-0.9em
\gS{,}
\kern-0.3em
\gS{\mathfrak{p}_{a}=}
\kern-0.3em
\begin{tikzpicture}[baseline=(A),outer sep=0pt,inner sep=0pt]
\node (A) at (0,0) {};
\twoeS{0}{0.7}{}{0}
\staru{0}{0}{a};

\begin{scope}[shift={(0.7,0)}]
\twostr{-0.4}{0.2}{};
\draw[line width=0.7pt, draw=black,dotted]  (-0.1,0.4) -- (0.1,0.4) ;
\twostr{0.4}{0.2}{};
\end{scope}

\end{tikzpicture}
\kern-1.3em
\gS{,a \in \{1,2\},}
\\
\gS{\mathfrak{q}_1=}
\kern-1.1em
\begin{tikzpicture}[baseline=(A),outer sep=0pt,inner sep=0pt]
\node (A) at (0,0) {};
\twoeS{0}{0.7}{}{0}
\twostr{-0.4}{0.2}{};
\draw[line width=0.7pt, draw=black,dotted]  (-0.1,0.4) -- (0.1,0.4) ;
\twostr{0.4}{0.2}{};

\begin{scope}[shift={(0.7,0)}]
\twostr{-0.4}{0.2}{};
\draw[line width=0.7pt, draw=black,dotted]  (-0.1,0.4) -- (0.1,0.4) ;
\twostr{0.4}{0.2}{};
\end{scope}

\end{tikzpicture}
\kern-1.3em
\gS{,}
\gS{\mathfrak{q}_a=}
\kern-0.5em
\begin{tikzpicture}[baseline=(A),outer sep=0pt,inner sep=0pt]
\node (A) at (0,0) {};

\twoeS{0}{0.7}{}{0.35};
\twoeSd{0}{0.7}{}{0.35};
\draw[line width=0.7pt, draw=black,dotted]  (0.35,0.15) --(0.35,-0.15) ;
\node[font=\tiny,align=center, fill=white, minimum size=2mm] at (0.35,0) {$a$};

\begin{scope}[rotate around={90:(0,0)}]
\twostr{-0.4}{0.2}{};
\draw[line width=0.7pt, draw=black,dotted]  (-0.1,0.4) -- (0.1,0.4) ;
\twostr{0.4}{0.2}{};
\end{scope}

\begin{scope}[shift={(0.7,0)}]
\begin{scope}[rotate around={-90:(0,0)}]
\twostr{-0.4}{0.2}{};
\draw[line width=0.7pt, draw=black,dotted]  (-0.1,0.4) -- (0.1,0.4) ;
\twostr{0.4}{0.2}{};
\end{scope}
\end{scope}

\end{tikzpicture}
\kern-0.5em
\gS{,}
\gS{\mathfrak{r}_a=}
\kern-0.5em
\begin{tikzpicture}[baseline=(A),outer sep=0pt,inner sep=0pt]
\node (A) at (0,0) {};

\twoeS{0}{0.7}{}{0.35};
\twoeSd{0}{0.7}{}{0.35};
\draw[line width=0.7pt, draw=black,dotted]  (0.35,0.15) --(0.35,-0.15) ;
\node[font=\tiny,align=center, fill=white, minimum size=2mm] at (0.35,0) {$a$};

\begin{scope}[rotate around={90:(0,0)}]
\twostr{-0.4}{0.2}{};
\draw[line width=0.7pt, draw=black,dotted]  (-0.1,0.4) -- (0.1,0.4) ;
\twostr{0.4}{0.2}{};
\end{scope}

\begin{scope}[shift={(0.7,0)}]
\begin{scope}[rotate around={-90:(0,0)}]
\twostr{-0.4}{0.2}{};
\draw[line width=0.7pt, draw=black,dotted]  (-0.1,0.4) -- (0.1,0.4) ;
\twostr{0.4}{0.2}{};
\end{scope}
\end{scope}

\staru{0.7}{0.5}{1};

\begin{scope}[rotate around={180:(0.7,-0.5)}]
\staru{0.7}{-0.5}{2};
\end{scope}
\twoeARB{0.7}{-0.5}{0.7}{0.5}{}{above=1pt};
\end{tikzpicture}
\kern-0.5em
\gS{, 3\leq a \leq 2p-1 \in 2\ZZ+1,}
\\
\gS{\mathfrak{s}_{abc}=}
\kern-0.5em
\begin{tikzpicture}[baseline=(A),outer sep=0pt,inner sep=0pt]
\node (A) at (0,0) {};

\twoeS{0}{0.7}{}{0.35};
\twoeSd{0}{0.7}{}{0.35};
\draw[line width=0.7pt, draw=black,dotted]  (0.35,0.15) --(0.35,-0.15) ;
\node[font=\tiny,align=center, fill=white, minimum size=2mm] at (0.35,0) {$a$};

\begin{scope}[rotate around={90:(0,0)}]
\twostr{-0.4}{0.2}{};
\draw[line width=0.7pt, draw=black,dotted]  (-0.1,0.4) -- (0.1,0.4) ;
\twostr{0.4}{0.2}{};
\end{scope}

\begin{scope}[shift={(0.7,0)}]
\begin{scope}[rotate around={-90:(0,0)}]
\twostr{-0.4}{0.2}{};
\draw[line width=0.7pt, draw=black,dotted]  (-0.1,0.4) -- (0.1,0.4) ;
\twostr{0.4}{0.2}{};
\end{scope}
\end{scope}

\staru{0.7}{0.5}{c};

\staru{0}{0.5}{b};
\twoeARB{0}{0}{0}{0.5}{}{above=1pt};
\twoeARB{0.7}{0}{0.7}{0.5}{}{above=1pt};
\end{tikzpicture}
\kern-0.5em
\gS{, 4\leq a \leq 2p \in 2\ZZ, bc \in \{11,12,22\},}
\\
\gS{\mathfrak{t}_a=}
\kern-0.5em
\begin{tikzpicture}[baseline=(A),outer sep=0pt,inner sep=0pt]
\node (A) at (0,0) {};

\twoeS{0}{0.7}{}{0.35};
\twoeSd{0}{0.7}{}{0.35};
\draw[line width=0.7pt, draw=black,dotted]  (0.35,0.15) --(0.35,-0.15) ;
\node[font=\tiny,align=center, fill=white, minimum size=2mm] at (0.35,0) {$a$};

\begin{scope}[rotate around={90:(0,0)}]
\twostr{-0.4}{0.2}{};
\draw[line width=0.7pt, draw=black,dotted]  (-0.1,0.4) -- (0.1,0.4) ;
\twostr{0.4}{0.2}{};
\end{scope}

\begin{scope}[shift={(0.7,0)}]
\begin{scope}[rotate around={-90:(0,0)}]
\twostr{-0.4}{0.2}{};
\draw[line width=0.7pt, draw=black,dotted]  (-0.1,0.4) -- (0.1,0.4) ;
\twostr{0.4}{0.2}{};
\end{scope}
\end{scope}

\staru{0.7}{0.5}{1};

\begin{scope}[rotate around={180:(0.7,-0.5)}]
\staru{0.7}{-0.5}{2};
\end{scope}
\twoeARB{0.7}{-0.5}{0.7}{0.5}{}{above=1pt};

\staru{0}{0.5}{1};

\begin{scope}[rotate around={180:(0,-0.5)}]
\staru{0}{-0.5}{2};
\end{scope}
\twoeARB{0}{-0.5}{0}{0.5}{}{above=1pt};
\end{tikzpicture}
\kern-0.5em
\gS{, 5\leq a \leq 2p-1 \in 2\ZZ+1.}
\end{center}

\end{proposition}

 \begin{lemma}\label{le:W3}
 Denote by $x$ the images of the $c$ in $\CC[V_1]^U$ and by $y$ those in $\CC[V_3]^{U^o}$.  The algebra  $A_3=(\CC[V_1]^U\otimes\CC[V_3]^{U^o})^T$ is minimally generated by $x^*_2,  y^\triangle_{2p+1}$ and the entries on and above the diagonal of the matrices
 \setlength{\arraycolsep}{0.1mm}
 \renewcommand{\arraystretch}{1.3}
$$
 f:=y^\triangle_1
 \cdot
 \begin{bmatrix}
{x^{(1)}_1}^2 & x^{(1)}_1 x^{(2)}_1   & x^{(1)}_1 x_p \\
x^{(1)}_1 x^{(2)}_1 & {x^{(2)}_1}^2 & x^{(2)}_1 x_p \\
x^{(1)}_1 x_p & x^{(2)}_1 x_p & x_p^2
 \end{bmatrix}
  ,
  \
g:=y^\triangle_2
 \cdot
 \begin{bmatrix}
{x^{(1)}_{p+1}}^2 & x^{(1)}_{p+1}x^{(2)}_{p+1} & x^{(1)}_{p+1}x_1^* \\
x^{(1)}_{p+1}x^{(2)}_{p+1} &  {x^{(2)}_{p+1}}^2 & x^{(2)}_{p+1} x_1^* \\
x^{(1)}_{p+1}x_1^* & x^{(2)}_{p+1} x_1^*  & {x_1^*}^2
 \end{bmatrix}
 ,
 $$
   $$
h^{(k)}:=
y^\triangle_{2k+1}
 \cdot
 \begin{bmatrix}
 x_{p-k}^2 &
 x_{p-k} x^*_{p+2-k} \\
 x_{p-k} x^*_{p+2-k} &  {x^*_{p+2-k}}^2
 \end{bmatrix}
,
k=1,\ldots,p-1,
$$
$$
 i^{(k)}:=
 y^\triangle_{2k+2}
 \cdot
  \begin{bmatrix}
 {x^{(1)}_{p+1-k}}^2 &
 x^{(1)}_{p+1-k} x^{(2)}_{p+1-k} \\
 x^{(1)}_{p+1-k} x^{(2)}_{p+1-k} & {x^{(2)}_{p+1-k}}^2
 \end{bmatrix}
 ,
 k=1,\ldots,p-1.
 $$
 The ideal of relations is generated  by all $2\times 2$-minors  of the matrices $f,g,h^{(k)}, i^{(k)}$ with at most one entry below the diagonal and in addition
$$
f_{a3}g_{b3}-p f_{a2}g_{1b} +p f_{1a} g_{2b}, \quad a,b \in \{1,2,3\}.
$$ 
 \end{lemma}

 \begin{proof}
We proceed exactly as in Lemmata~\ref{le:W1} and~\ref{le:W2}.
\end{proof}

 \begin{proof}[Proof of Prop.~\ref{prop:W3}] We have 
\begin{align*}
 A_3 = 
\CC[x^*_2, y^\triangle_{2p+1}]  &  \otimes  \CC[y^\triangle_{1},y^\triangle_{2}, x^{(1)}_{1},x^{(2)}_{1},x_p,x^{(1)}_{p+1},x^{(2)}_{p+1},x^*_1]^T  \\
 & \otimes  \bigotimes_{k=1}^{p-1}\CC[y^\triangle_{2k+1}, x_{p-k},x^*_{p+2-k}]^T   \otimes \bigotimes_{k=1}^{p-1}\CC[y^\triangle_{2k+2}, x^{(1)}_{p+1-k},x^{(2)}_{p+1-k}]^T.
    \end{align*}
 With variables $t$, $s$ and $r_i$ corresponding to copies of $\BigWedge^2V$, $S^2V$ and $\BigWedge^{n-1}V$ respectively, the first two algebras and the $k$-th factors of the third and fourth algebra have respective Hilbert series 
$$
{\scriptstyle
\frac{1}{(1-tr_1r_2)(1-s^{2p+1})},
\frac {1-{s}^{3}{t}^{2p}r_{{1}}^{2}r_{{2}}^{2}}{ (1- s{t}^{2p})  (1- sr_{{2}}{t}^{p}) (1- sr_{{2}}^{2}) (1- sr_{{1}}{t}^{p})(1- sr_{{1}}r_{{2}}) (1- sr_{1}^{2})}
}
$$
$$
{\scriptstyle
\frac{1-s^{4k+2}t^{4p-4k+2}r_1^2r_2^2}{(1-s^{2k+1}t^{2p-2k})(1-s^{2k+1}t^{2p-2k+1}r_1r_2)(1-s^{2k+1}t^{2p-2k+2}r_1^2r_2^2)},}
$$
$$
{\scriptstyle
\frac{1-s^{4k+4}t^{4p-4k}r_1^2r_2^2}{(1-s^{2k+2}t^{2p-2k}r^2_1)(1-s^{2k+2}t^{2p-2k}r_1r_2)(1-s^{2k+2}t^{2p-2k}r^2_2)
}
}
$$
By~\cite[Proof of Thm. 0.4]{shmel}, a minimal system of generators has the same set of multidegrees as the system of generators of $A_3$ from Lemma~\ref{le:W3}, where those generators must be checked for reducibility and in case omitted, for which $A_3$ has a syzygy of smaller or equal multidegree. This is the case for all entries of $g$ and for $h^{(1)}_{2,2}$, where invariants having the same multidegree can be shown to be reducible by applying Pl\"ucker relations to two-edges. 
We arrive at the proposed minimal system of generators of $\CC[W_3]^{\SL_{2p+1}}$. Concerning the syzygies, the Hilbert series tells us to search for ones corresponding to $(1-s^{4k+2}t^{4p-4k+2}r_1^2r_2^2)$ and $(1-s^{4k+4}t^{4p-4k}r_1^2r_2^2)$. Consider the reducible graphs
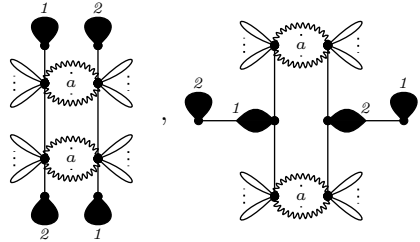
\begin{center}
\begin{tikzpicture}[baseline=(A),outer sep=0pt,inner sep=0pt]
\node (A) at (0,-0.5) {};

\twoeS{0}{0.7}{}{0.35};
\twoeSd{0}{0.7}{}{0.35};
\draw[line width=0.7pt, draw=black,dotted]  (0.35,0.15) --(0.35,-0.15) ;
\node[font=\tiny,align=center, fill=white, minimum size=2mm] at (0.35,0) {$a$};

\begin{scope}[rotate around={90:(0,0)}]
\twostr{-0.4}{0.2}{};
\draw[line width=0.7pt, draw=black,dotted]  (-0.1,0.4) -- (0.1,0.4) ;
\twostr{0.4}{0.2}{};
\end{scope}

\begin{scope}[shift={(0.7,0)}]
\begin{scope}[rotate around={-90:(0,0)}]
\twostr{-0.4}{0.2}{};
\draw[line width=0.7pt, draw=black,dotted]  (-0.1,0.4) -- (0.1,0.4) ;
\twostr{0.4}{0.2}{};
\end{scope}
\end{scope}

\staru{0.7}{0.5}{2};

\staru{0}{0.5}{1};
\twoeARB{0}{0}{0}{0.5}{}{above=1pt};
\twoeARB{0.7}{0}{0.7}{0.5}{}{above=1pt};

\begin{scope}[shift={(0,-1)}]
\twoeS{0}{0.7}{}{0.35};
\twoeSd{0}{0.7}{}{0.35};
\draw[line width=0.7pt, draw=black,dotted]  (0.35,0.15) --(0.35,-0.15) ;
\node[font=\tiny,align=center, fill=white, minimum size=2mm] at (0.35,0) {$a$};

\begin{scope}[rotate around={90:(0,0)}]
\twostr{-0.4}{0.2}{};
\draw[line width=0.7pt, draw=black,dotted]  (-0.1,0.4) -- (0.1,0.4) ;
\twostr{0.4}{0.2}{};
\end{scope}

\begin{scope}[shift={(0.7,0)}]
\begin{scope}[rotate around={-90:(0,0)}]
\twostr{-0.4}{0.2}{};
\draw[line width=0.7pt, draw=black,dotted]  (-0.1,0.4) -- (0.1,0.4) ;
\twostr{0.4}{0.2}{};
\end{scope}
\end{scope}
\end{scope}

\begin{scope}[rotate around={180:(0.7,-1.5)}]
\staru{0.7}{-1.5}{1};
\end{scope}
\begin{scope}[rotate around={180:(0,-1.5)}]
\staru{0}{-1.5}{2};
\end{scope}
\twoeARB{0}{-1}{0}{-1.5}{}{above=1pt};
\twoeARB{0.7}{-1}{0.7}{-1.5}{}{above=1pt};

\twoeARB{0}{0}{0}{-1}{}{above=1pt};
\twoeARB{0.7}{0}{0.7}{-1}{}{above=1pt};
\end{tikzpicture}
\gS{,}
\begin{tikzpicture}[baseline=(A),outer sep=0pt,inner sep=0pt]
\node (A) at (0,-1) {};

\twoeS{0}{0.7}{}{0.35};
\twoeSd{0}{0.7}{}{0.35};
\draw[line width=0.7pt, draw=black,dotted]  (0.35,0.15) --(0.35,-0.15) ;
\node[font=\tiny,align=center, fill=white, minimum size=2mm] at (0.35,0) {$a$};

\begin{scope}[rotate around={90:(0,0)}]
\twostr{-0.4}{0.2}{};
\draw[line width=0.7pt, draw=black,dotted]  (-0.1,0.4) -- (0.1,0.4) ;
\twostr{0.4}{0.2}{};
\end{scope}

\begin{scope}[shift={(0.7,0)}]
\begin{scope}[rotate around={-90:(0,0)}]
\twostr{-0.4}{0.2}{};
\draw[line width=0.7pt, draw=black,dotted]  (-0.1,0.4) -- (0.1,0.4) ;
\twostr{0.4}{0.2}{};
\end{scope}
\end{scope}

\twoeARB{0}{0}{0}{-1}{}{above=1pt};
\twoeARB{0.7}{0}{0.7}{-1}{}{above=1pt};
\twoeARB{0}{-2}{0}{-1}{}{above=1pt};
\twoeARB{0.7}{-2}{0.7}{-1}{}{above=1pt};

\begin{scope}[shift={(0,-2)}]
\twoeS{0}{0.7}{}{0.35};
\twoeSd{0}{0.7}{}{0.35};
\draw[line width=0.7pt, draw=black,dotted]  (0.35,0.15) --(0.35,-0.15) ;
\node[font=\tiny,align=center, fill=white, minimum size=2mm] at (0.35,0) {$a$};

\begin{scope}[rotate around={90:(0,0)}]
\twostr{-0.4}{0.2}{};
\draw[line width=0.7pt, draw=black,dotted]  (-0.1,0.4) -- (0.1,0.4) ;
\twostr{0.4}{0.2}{};
\end{scope}

\begin{scope}[shift={(0.7,0)}]
\begin{scope}[rotate around={-90:(0,0)}]
\twostr{-0.4}{0.2}{};
\draw[line width=0.7pt, draw=black,dotted]  (-0.1,0.4) -- (0.1,0.4) ;
\twostr{0.4}{0.2}{};
\end{scope}
\end{scope}
\end{scope}

\staru{1.7}{-1}{1};

\staru{-1}{-1}{2};

\draw[line width=0.5pt, draw=black,fill=black]  (0,-1) .. controls (-0.2,-0.8) and (-0.3,-0.8) .. (-0.5,-1) ;
\draw[line width=0.5pt, draw=black,fill=black]  (0,-1) .. controls (-0.2,-1.2) and (-0.3,-1.2) .. (-0.5,-1) ;
\draw[line width=0.5pt, draw=black]  (0,-1) -- (-1,-1) node[font=\tiny,midway, above=3pt] {$\ivE{1}$};

\draw[line width=0.5pt, draw=black,fill=black]  (1.2,-1) .. controls (1,-0.8)  and (0.9,-0.8) .. (0.7,-1) ;
\draw[line width=0.5pt, draw=black,fill=black]  (0.7,-1) .. controls (0.9,-1.2) and (1,-1.2) .. (1.2,-1) ;
\draw[line width=0.5pt, draw=black]  (0.7,-1) -- (1.7,-1) node[font=\tiny,midway, above=3pt] {$\ivE{2}$};

\end{tikzpicture} 
\end{center}
for $a=2k+1$ in the left and $a=2k+2$ in the right one. For both types, there are \emph{two ways} to transform them into a sum of disconnected ones, so we get a syzygy for each of them. These two ways are the following: for the left one, consider the Pl\"ucker relation applied to the two-edges connecting the upper and lower part, but for the first way pulling both to the upper vertices and for the second way pulling one to the upper and one to the lower vertices. This results in two equivalent sums of reducible graphs,  containing $\mathfrak{r}_a^2$ and $\mathfrak{q}_a \mathfrak{t}_a$ respectively. We do not list the other summands but remark that this is exactly the part of the relation that reflects in the algebra $A_3$ as the determinant of $h^{(k)}$. 

For the right graph, consider for the first way the Pl\"ucker relation applied to the two-edges connecting the upper and middle part, and for the second way, which is more important for us, the Pl\"ucker relation applied to the horizontal non-looping $2p$-edges, both pulled towards the middle. The resulting syzygy contains - resulting from the second way - $\mathfrak{s}_{a12}^2$ and $\mathfrak{s}_{a11}\mathfrak{s}_{a22}$, which again are exactly the part of the relation reflecting in $A_3$, but now as the determinant of $i^{(k)}$. 

Finally, we apply the Crosshair-sieve-algorithm~\ref{algCS} to show that $\CC[W_3]^{\SL_n}$ is a complete intersection. This is done by choosing the $\mathfrak{r}_a^2$ and $\mathfrak{s}_{a12}^2$ as leading monomials. The  sieves here can be viewed as part-matrices up and above the diagonal,  containing in the $(i,j)$-th entry the degree of $\mathfrak{r}_{2+i}\mathfrak{r}_{2+j}$, $\mathfrak{r}_{2+i}\mathfrak{s}_{(2+j)12}$, $\mathfrak{s}_{(2+i)12}\mathfrak{r}_{2+j}$ and $\mathfrak{s}_{(2+i)12}\mathfrak{s}_{(2+j)12}$ for $(i,j) \in \{(2\ZZ+1)^2,(2\ZZ+1)\times 2\ZZ, 2\ZZ \times (2\ZZ+1), 2\ZZ^2\}$ respectively.
By the multidegrees of the syzygies, monomials that may occur in one relation lie on a counterdiagonal containing an entry on the diagonal (corresponding to a square monomial), as is shown in the following example picture:
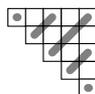
\begin{center}
\begin{tikzpicture}
\node[matrix,  nodes={fill=white,draw,minimum size=2mm},row sep=-\pgflinewidth,column sep=-\pgflinewidth] (Leva) at (0,0)
{
\node(a1) {}; & \node (a12) {}; & \node (a13) {}; & \node (a14) {}; & \node (a15) {}; \\
  & \node(a22) {}; & \node (a23) {}; & \node (a24) {}; & \node (a25) {};   \\
& & \node(a33) {};  & \node {}; & \node (a35) {}; \\
&& & \node (a44) {}; & \node (a12) {}; \\
&&&& \node (a55) {}; \\
};
\draw[line width=3pt, draw=black,opacity=0.5,cap=round]  (a1.center) --(a1.center)  ;
\draw[line width=3pt, draw=black,opacity=0.5,cap=round]  (a22.center)--(a13.center) ;
\draw[line width=3pt, draw=black,opacity=0.5,cap=round]  (a33.center)  --(a15.center)  ;
\draw[line width=3pt, draw=black,opacity=0.5,cap=round]  (a44.center)--(a35.center) ;
\draw[line width=3pt, draw=black,opacity=0.5,cap=round]  (a55.center)--(a55.center) ;
\end{tikzpicture}
\end{center}
Now as the targeted monomials lie on the diagonal, $\mathcal{S}_1$ filters out $\mathfrak{r}_3^2$ and $\mathfrak{s}_{(2p)12}^2$, then $\mathcal{S}_2$ filters out $\mathfrak{s}_{412}^2$ and $\mathfrak{r}_{(2p-1)}^2$ et cetera. So we found a Gr\"obner basis for the ideal of relations of $\CC[W_3]^{\SL_n}$ and it is a complete intersection. 

\end{proof}

\end{document}